\documentclass[12pt,twoside]{article}

\usepackage[T1]{fontenc}
\usepackage[latin1]{inputenc}
\usepackage[english]{babel}
\usepackage[babel]{csquotes}

\usepackage{cite}

\usepackage{amssymb}
\usepackage{amsmath}
\usepackage{amsthm}
\usepackage{latexsym}
\usepackage{graphicx}
\usepackage{mathrsfs}
\usepackage{mathrsfs}

\usepackage[usenames,dvipsnames]{color}

\definecolor{viola}{rgb}{0.3,0,0.7}
\definecolor{ciclamino}{rgb}{0.5,0,0.5}

\def\luca #1{{\color{red}#1}}
\def\luca #1{#1}

\theoremstyle{plain}
\newtheorem{thm}{Theorem}[section]

\newtheorem{lem}[thm]{Lemma}
\newtheorem{prop}[thm]{Proposition}

\theoremstyle{definition}
\newtheorem{rmk}[thm]{Remark}

\def\En{\mathbb{N}}

\def\Ar{\mathbb{R}}

\def\Pi{\mathbb{P}}

\def\eps{\varepsilon}
\def\beq{\begin{equation}}
\def\eeq{\end{equation}}

\def\rarr{\rightarrow}

\def\l|{\left\|}
\def\r|{\right\|}
\def\L2{L^2(D)}
\def\H1{H^1(D)}
\def\gH1{H^1_{0}(D)}
\def\nL2#1{\l|#1\r|_{\L2}}
\def\nH1#1{\l|#1\r|_{\H1}}
\def\rarrw{\rightharpoonup}
\def\Div{\mathop{\rm div}}
\def\gL2{L^2(\Gamma)}
\def\ngL2#1{\l|#1\r|_{\gL2}}
\def\weakstar{\stackrel{*}{\rightharpoonup}}
\def\embed{\hookrightarrow}
\def\cembed{\stackrel{c}{\hookrightarrow}}

\def\f{\mathcal{F}}
\def\rsz{\mathcal{R}}
\def\exval{\mathbb{E}}
\def\cL{\mathscr{L}}

\begin{document}
\begin{center}
{\huge\rm Well-posedness for a class of doubly nonlinear\\[0.2cm] stochastic PDEs of divergence type\/\footnote{{\bf 
Acknowledgments.}\quad\rm The author is very grateful to Carlo Marinelli for suggesting the problem
and for his expert advice and fundamental support throughout this project.}}
\\[1cm]
{\large\sc Luca Scarpa}\\[0.15cm]
{\normalsize e-mail: {\tt luca.scarpa.15@ucl.ac.uk}}\\[.15cm]
{\small Department of Mathematics, University College London}\\[.15cm]
{\small Gower Street, London WC1E 6BT, United Kingdom}\\[1.25cm]
\end{center}       

\begin{abstract}
We prove well-posedness for doubly nonlinear parabolic stochastic partial differential equations
of the form $dX_t-\Div\gamma(\nabla X_t)\,dt+\beta(X_t)\,dt\ni B(t,X_t)\,dW_t$, where $\gamma$
and $\beta$
are the two nonlinearities, assumed to be multivalued maximal monotone operators everywhere defined on $\Ar^d$ and $\Ar$
respectively, and $W$ is a cylindrical Wiener process. Using variational techniques,
suitable uniform estimates (both pathwise and in expectation) and some compactness results,
well-posedness is proved 
under the classical Leray-Lions conditions on $\gamma$ and with no restrictive smoothness or growth assumptions on
$\beta$. The operator $B$ is assumed to be Hilbert-Schmidt and to satisfy some classical
Lipschitz conditions in the second variable. \\[.5cm]
\end{abstract}

\pagestyle{myheadings}
\newcommand\testopari{\sc Luca Scarpa}
\newcommand\testodispari{\sc Well-posedness for doubly nonlinear SPDEs of divergence type}
\markboth{\testodispari}{\testopari}


\thispagestyle{empty}

\section{Introduction}
\setcounter{equation}{0}
\label{intro}

In this work, we consider the boundary value problem with homogeneous Dirichlet conditions
associated to a doubly nonlinear parabolic stochastic partial differential equation
on an smooth bounded domain $D\subseteq\Ar^d$ of the type
\begin{gather}
  \label{eq}
  dX_t-\Div\gamma(\nabla X_t)\,dt+\beta(X_t)\,dt\ni B(t, X_t)\,dW_t \quad\text{in } D\times(0,T)\,,\\
  \label{init}
  X(0)=X_0 \quad\text{in } D\,,\\
  \label{bound}
  X=0 \quad\text{on } \partial D\times(0,T)\,,
\end{gather}
where $\gamma$ and
$\beta$ are two maximal monotone operators everywhere defined on $\Ar^d$ and $\Ar$, respectively, 
$W$ is a cylindrical Wiener process, and $B$ is a random time-dependent
Hilbert-Schmidt operator
(we will state the complete assumptions on the data in the next section). We
prove existence of global solutions
as well as a continuous dependence result using variational techniques
(see e.g.~the classical works \cite{pardoux1, pardoux2, KR-spde}
about the variational approach to SPDEs).

The problem \eqref{eq}--\eqref{bound} is very interesting 
from the mathematical point of view:
as a matter of fact, the equation presents two strong nonlinearities.
The first one is represented by $\gamma$ within
the divergence operator: in this case, we will need to assume some classical growth assumptions (the 
so-called Leray-Lions conditions) in 
order to recover a suitable coercivity on a natural Sobolev space. 
The other nonlinearity is represented by the operator $\beta$: this is treated as generally as possible, with no
restriction on the growth and regularity. Because of this generality,
the concept of solution and the appropriate estimates are more difficult to achieve, as we will see.
We point also out that dealing with maximal monotone graphs makes our analysis absolutely exhaustive.
As a matter of fact, in this way we include in our treatment any continuous increasing function $\beta$ 
(with any order of growth), as well as every increasing function with a countable number 
of jumps: indeed, it is a standard matter to see that if $\beta$ is an increasing function on $\Ar$ with
jumps in $\{x_n\}_{n\in\En}$, one can obtain a maximal monotone
graph by setting $\beta(x_n)=[\beta_-(x_n), \beta_+(x_n)]$.
Finally, very mild assumptions on the noise are required, so that our results fit to any reasonable
random time-dependent Hilbert-Schmidt operator $B$; in the case of multiplicative noise,
only classical Lipschitz continuity hypotheses are in order.

The noteworthy feature of this paper is that
problem \eqref{eq}--\eqref{bound} is very general and embraces a wide variety of specific sub-problems
which are interesting on their own: consequently, we provide with our treatment a unifying analysis to
several cases of parabolic SPDEs.
Let us mention now about some of these and the main
related literature.

If $\gamma$ is the identity on $\Ar^d$, the resulting equation is the classical
semilinear SPDE driven by the Laplace operator $dX-\Delta X\,dt+\beta(X)\,dt\ni B\,dW_t$, which has been widely studied.
For example,
in \cite{mar-scar}, global existence results of solutions are provided in the semilinear case, 
with the laplacian being generalized to any suitable linear operator: here, the idea is to doubly approximate the problem, in order
to recover more regularity on $\beta$ and $B$, to find then appropriate estimates on the approximated solutions and 
finally to pass to the limit in the equation.
Moreover, in \cite{daprato}, mild solutions are obtained under the strong hypotheses that $\beta$ is a polynomial
of odd degree $m>1$ and $B$ can be written as $(-\Delta)^{-\frac{s}{2}}$ for a suitable $s$;
in \cite{barbu}, existence of mild solutions is proved with no restrictive hypotheses on the growth of $\beta$,
but imposing some strong continuity assumptions on the stochastic convolution.
In \cite{marinelli}, well-posedness is established for the semilinear problem in a $L^q$ setting,
with $\beta$ having polynomial growth.

If $\gamma$ is the monotone function on $\Ar^d$ given by $\gamma(x)=|x|^{p-2}x$, $x\in\Ar^d$, for a certain $p\geq2$, then 
the term represented by the divergence in \eqref{eq} is the usual $p$-laplacian: in this case, our equation becomes
$dX-\Delta_pX\,dt+\beta(X)\,dt\ni B\,dW_t$, where $\Delta_p\cdot:=\Div(|\nabla\cdot|^{p-2}\nabla\cdot)$.
This problem is far more interesting and complex than the semilinear case since
$-\Delta_p$ is nonlinear for any $p>2$ and consequently \eqref{eq} becomes doubly nonlinear in turn.
Among the great literature dealing with this problem, we can mention \cite{liu-plap} for example, where the stochastic
$p$-Laplace equation is studied in the singular case $p\in[1,2)$, and \cite{liu-plap2} as well.

Let us now briefly outline the structure of the paper and the results that we present.

In section \ref{results}, we state the precise assumptions of the work and we accurately describe the general setting:
here, the main hypotheses are stated and the variational setting is presented. 
Furthermore, we outline the four main results: the first theorem 
ensures that problem \eqref{eq}--\eqref{bound} admits global solutions in a suitable weak variational way
in the case of additive noise,
the second one is the very natural continuous dependence property with respect to
the initial datum and $B$, the third is the
existence result in case of multiplicative noise and the last one states the
continuous dependence property with respect to the initial datum in case of multiplicative noise.

Section \ref{exist} contains the proof of the existence theorem with additive noise: the main idea is to introduce two
approximations on the problem. The first approximation depends on a parameter~$\lambda$ and
it is made on the maximal monotone operators $\beta$ and $\gamma$, 
considering the Yosida approximations, as usual; moreover,
a correction term is added in order to recover a suitable coercivity when $\lambda$ is fixed,
and that is going to vanish when taking the limit as $\lambda\searrow0$.
The second approximation depends on a parameter $\eps$ and is 
made on the operator $B$
in order to gain more regularity on the noise.
The double approximation is very similar to the one performed in \cite{mar-scar}.
The general idea is that given a fixed approximation in $\eps$, the approximated noise is
regular enough to allow us to pass to the limit pathwise in $\lambda$: once this first step is carried out,
suitable probability estimates allow us to pass to the limit also in $\eps$.
More specifically,
the proof of existence consists in obtaining uniform estimates on the approximated solutions,
independently of the approximations, and then passing to the limit in the approximated problem.
To this purpose, we will recover pathwise estimates which are uniform in $\lambda$ (but not in $\eps$),
and global estimates also in expectation which are uniform both in $\lambda$ and in $\eps$.
The passage to the limit is carried out in two steps:
the first is on $\lambda$ and it is made pathwise, while the second is made on $\eps$ and
is made globally also in probability.
The main idea is to use It\^o's formula and some sharp testings to obtain
$L^1$ estimates on the nonlinear terms in $\beta$ and rely on the Dunford-Pettis theorem to recover a 
weak compactness, being inspired in this sense by some calculations performed in \cite{barbu, mar-scar}.

Section \ref{cont_dep} is devoted to proving the continuous dependence result for the additive noise case,
which easily follows from the definition of solution itself and a generalized It\^o formula, which is accurately proved
in the Appendix \ref{B_app}.

Section \ref{mult} contains the proof of the main result, which ensures that the problem with 
multiplicative noise is well-posed: here, we build the global solutions step-by-step, proving at each iteration
accurate contraction estimates and using classical fixed-points arguments. The continuous dependence follows from
the generalized version of It\^o's formula contained in Appendix \ref{B_app}.

The appendixes \ref{A_app} and \ref{B_app} contain a version of a variational integration-by-parts formula and
the generalized It\^o formula, which are widely used
throughout the paper: the first one is made pathwise and it is used when passing to the limit on $\lambda$ in order
to identify the limit of the nonlinearity in $\gamma$, while the second is a direct generalization of the classical 
It\^o formula in a variational setting, and it is needed in the passage to the limit on $\eps$ and in the 
proof of the continuous dependence. The idea of the proof is to identify
accurate approximations on the processes which have to satisfy appropriate conditions, such as linearity, smoothness
properties and suitable asymptotical behaviours: in this sense, appropriate elliptic approximations are performed.


\section{Setting and main results}
\setcounter{equation}{0}
\label{results}

In this section we state the precise assumptions on the data of the problem and the concept of solution. Moreover, 
we present the main results which will be proved in the subsequent sections.

In the entire work, $\left(\Omega, \f, \mathbb{F}, \Pi\right)$ is a filtered probability space,
where the filtration $\mathbb{F}=\left(\f_t\right)_{t\in[0,T]}$
is assumed to satisfy the so-called "usual conditions" (i.e.~it is saturated and right continuous)
and $T>0$ is the fixed final time;
moreover, $D\subseteq\Ar^d$ is a smooth bounded domain and $Q:=D\times(0,T)$ is
the corresponding space-time cylinder. Furthermore, we set
\beq
  \label{spaces}
  H:=L^2(D)
\eeq
and we
use the symbol $(\cdot,\cdot)$ for the standard inner product of $H$. Moreover,
if $U$ is a Banach space,
we simply write $L^p(\Omega; U)$ (without specifying the $\sigma$-algebra) 
to indicate the usual class of Bochner-integrable functions $L^p(\Omega, \f, \Pi; U)$; when we are referring
to the measure space with respect to a 
particular $\sigma$-algebra of the filtration, we write explicitly $L^p(\Omega, \f_t, \Pi; U)$ for any given $t\in[0,T]$.
The symbol $C^0_w([0,T]; U)$ denotes the space of continuous functions from $[0,T]$ to
the space $U$ endowed with the weak topology.
Furthermore, if $U$ is a separable Hilbert Space, we will use the symbols $\cL(U,H)$ and $\cL_2(U,H)$ to indicate 
the spaces of the linear continuous operators and Hilbert-Schmidt operators from $U$ to $H$, respectively. 

We write "$\cdot$" for the usual scalar product in $\Ar^d$, while the symbols $\embed$ and $\cembed$
indicate a continuous and a compact-continuous inclusion between Banach spaces, respectively.
Moreover, for any constant appearing in the paper, we indicate in the subscript any quantity on which the constant depends: 
for example, we may use the notation $C_{a,b}$ to stress that the constant $C$ only depends on $a$ and $b$.
Finally, for any $x,y\in\Ar$, the notation $x\lesssim y$ means that there exists a positive constant $C$ such that $x\leq Cy$;
clearly, we write $x\lesssim_{a,b}y$ for $x\leq C_{a,b}y$, for any quantity $a,b\in\Ar$.

We can now specify the main hypotheses of our work. 
First of all, we introduce
\begin{gather}
  \label{gamma}
  \gamma:\Ar^d\rarr2^{\Ar^d} \quad\text{maximal monotone}\,, \quad D(\gamma)=\Ar^d\,, \quad 0\in\gamma(0)\\
  \label{beta}
  \beta:\Ar\rarr2^\Ar \quad\text{maximal monotone}\,, \quad D(\beta)=\Ar\,, \quad 0\in\beta(0)\\
  \label{W}
  W\quad\text{cylindrical Wiener process on } U\,,
\end{gather}
where $U$ is a suitable separable 
Hilbert space. Now,
thanks to definition \eqref{beta}, the function
\beq
  \label{j}
  j:\Ar\rarr[0,+\infty) \quad\text{proper, convex, lower semicontinuous}\,, \quad \partial j=\beta\,, \quad j(0)=0
\eeq
is well defined; furthermore, we make the assumption that also $\gamma$ is a subdifferential, i.e.~that there exists
\beq
  \label{k}
  k:\Ar^d\rarr[0,+\infty) \quad\text{proper, convex, lower semicontinuous}\,, \quad \partial k=\gamma\,,\quad k(0)=0\,.
\eeq
We denote by $k^*$ and $j^*$ the convex conjugate functions of $k$ and $j$, respectively, i.e.
\beq
  \label{k_star}
  k^*:\Ar^d\rarr[0,+\infty]\,, \quad k^*(r):=\sup_{y\in\Ar^d}\{r\cdot y-k(y)\}\,,
\eeq
\beq
  \label{j_star}
  j^*:\Ar\rarr[0,+\infty]\,, \quad j^*(r):=\sup_{y\in\Ar}\{ry-j(y)\}\,.
\eeq
The following facts from convex analysis are well-known (see for example \cite{barbu_monot, brezis}):
\begin{gather}
  \label{convex2}
  k(z)+k^*(s)=s\cdot z \iff s\in\partial k(z)\,, \qquad j(y)+j^*(r)=ry \iff r\in\partial j(y)\,,\\
  \label{convex3}
  k(z)+k^*(s)\geq s\cdot z\,, \qquad j(y)+j^*(r)\geq ry \qquad\text{for all } y,r\in\Ar\,,\, z,s\in\Ar^d\,.
\end{gather}
Throughout the paper, we will also assume that $j$ is even, i.e.
\beq
  \label{j_even}
  j(x)=j(-x) \quad\text{for every } x\in\Ar.
\eeq
\begin{rmk}
 Hypothesis \eqref{j_even} is needed in order to prove 
 the generalized It\^o formula for the solutions of our problem, which
 will be strongly used throughout the proofs. However, \eqref{j_even} can be weakened:
 the main point is that we only need $j$ to grow at the same rate both at $+\infty$ and at $-\infty$
 (cf.~\cite[p.~429]{barbu-prato-rock}).
 In order to simplify the treatment we assume \eqref{j_even}, but for sake of completeness
 we mention that we could have required a slightly weaker condition, namely
 \[
 \limsup_{|x|\rarr+\infty}\frac{j(x)}{j(-x)}<+\infty\,.
 \]
\end{rmk}
Now, for every $\delta\in(0,1)$, we introduce the resolvents and the Yosida approximations of $\gamma$ and $\beta$ as
\beq
  \label{res}
  J_\delta:=(I_d+\delta\gamma)^{-1}\,, \qquad R_{\delta}:=(I_1+\delta\beta)^{-1}\,,
\eeq 
\beq
  \label{yos}
  \gamma_\delta:=\frac{I_d-J_\delta}{\delta}\,, \qquad \beta_\delta:=\frac{I_1-R_\delta}{\delta}\,,
\eeq
where the symbol $I_m$ stands for the identity in $\Ar^m$ for any $m\in\En$.
Then, for every $\delta\in(0,1)$,
$J_\delta$, $R_\delta$, $\gamma_\delta$ and $\beta_\delta$ are single-valued, with the latter two being
$\frac{1}{\delta}$-Lipschitz continuous, and
\begin{gather}
  \label{prop_res1}
   \left|J_\delta x\right|\leq|x| \quad\text{for all } x\in\Ar^d\,,\qquad
  \left|R_\delta x\right|\leq|x| \quad\text{for all } x\in\Ar\,,\\
  \label{prop_res2}
  \gamma_\delta(x)\in\gamma\left(J_\delta x\right) \quad\text{for all } x\in\Ar^d\,, \qquad
  \beta_\delta(x)\in\beta\left(R_\delta x\right) \quad\text{for all } x\in\Ar
\end{gather}
(see for example \cite{barbu_semig, brezis}).

As we have anticipated, we need to make some assumptions on the growth of $\gamma$,
namely the so-called Leray-Lions conditions, which are widely required in the classical
literature on elliptic and parabolic PDEs
(the reader can refer here to \cite{bocc, bocc-gall, bocc2} for classical examples).
More in detail, we suppose that there are positive constants $K$, $D_1$, $D_2$ and an exponent $p\in[2,+\infty)$ such that
\begin{gather}
  \label{gamma1}
  \sup\{|y|: y\in\gamma(r)\}\leq D_1\left(1+|r|^{p-1}\right) \quad\text{for every } r\in\Ar^d\,,\\
  \label{gamma2}
  y\cdot r\geq K|r|^p-D_2 \quad\text{for every } r\in\Ar^d\,, \, y\in\gamma(r)\,.
\end{gather}
In the sequel, we will write $q:=\frac{p}{p-1}\in(1,2]$ for the conjugate exponent of $p$.

Finally, we set
\beq
  \label{V}
  V:=W^{1,p}_0(D)
\eeq 
and define the divergence operator in the
variational sense:
\beq
  \label{div}
  -\Div: L^q(D)^d\rarr V^*\,, \quad \left<-\Div u,v\right>
  :=\int_D{u\cdot\nabla v}\,,\quad u\in L^q(D)^d\,,\,v\in V\,,
\eeq
where we have used the the symbol $\left<\cdot,\cdot\right>$ for the duality pairing between $V$ and $V^*$. Here and
in the sequel, we make the natural identification
$H\cong H^*$, so that $H$ is continuously embedded in $V^*$: for every $u\in H$ and $v\in V$, we have
$\left<u,v\right>=(u,v)$.
Taking these remarks into account, we have
\beq
  V\cembed H\embed V^*\,,
\eeq
where the first inclusion is also dense.
Moreover, we set
\beq
  \label{V0}
  V_0:=H^{k}_0(D)\,, \quad k:=\left[\max\left\{\frac{d}{2}, 1+\frac{d}{2}-\frac{d}{p}\right\}\right]+1\,:
\eeq
note that with this particular choice of $k$, the classical results on Sobolev 
embeddings (see \cite[Thm.~1.5]{barbu_monot} and \cite[Thm.~219]{BG87}) ensure that
\[
  V_0\embed V \quad\text{densely}\,, \quad V_0\embed L^\infty(D)\,,
\]
so that we have
\beq
  \label{rel_V0}
  V_0\embed V\cap L^\infty(D)\,, \qquad
  V^*, L^1(D)\embed V_0^*\,.
\eeq

We can now state the four main results of the paper, which ensure that problem \eqref{eq}--\eqref{init}
is well-posed, both with additive and multiplicative noise.

\begin{thm}
  \label{th1}
  In the setting \eqref{spaces}--\eqref{rel_V0}, assume that
  \begin{gather}
  \label{ip1}
    X_0\in L^2\left(\Omega, \f_0, \Pi; H\right)\,,\\
  \label{ip2}
  B\in L^2\left(\Omega\times(0,T);\cL_2(U,H)\right) \quad\text{progressively measurable}\,,\\
  \label{ip3}
  \luca{\gamma \quad\text{is single-valued }\,;}
  \end{gather}
  then there exist
  \begin{gather}
    \label{sol1}
    X\in L^2\left(\Omega;L^\infty(0,T; H)\right)\cap L^p\left(\Omega\times(0,T); V\right)\,, \qquad
    X\in C^0_w\left([0,T]; H\right) \quad\Pi\text{-a.s.}\,,\\
    \label{sol2}
    \eta\in L^q\left(\Omega\times(0,T)\times D\right)^d\,,\\
    \label{sol3}
    \xi\in L^1\left(\Omega\times(0,T)\times D\right)\,,
  \end{gather}
  where
  $X$ and $\xi$ are predictable,
  $\eta$ is adapted, and the following relations hold:
  \begin{gather}
    \label{var_form}
    \begin{split}
    X(t)-\int_0^t{\Div\eta(s)\,ds}+\int_0^t{\xi(s)\,ds}=&X_0+\int_0^t{B(s)\,dW_s} \quad\text{in } L^1(D)\cap V^*\,,\\
    &\text{for every } t\in[0,T]\,, \quad \Pi\text{-almost surely}\,,
    \end{split}\\
    \label{incl_eta}
    \eta\in\gamma(\nabla X) \quad\text{a.e.~in } \Omega\times(0,T)\times D\,,\\
    \label{incl_xi}
    \xi\in\beta(X) \quad\text{a.e. in } \Omega\times(0,T)\times D\,,\\
    \label{j_integr}
    j(X)+j^*(\xi)\in L^1\left(\Omega\times(0,T)\times D\right)\,.
  \end{gather}
  Furthermore, if hypothesis \eqref{ip3} is not assumed, then the same conclusion is true replacing
  conditions \eqref{sol1} and \eqref{var_form} with, respectively,
  \beq
  \label{sol1_weak}
  X\in L^\infty\left(0,T;L^2(\Omega; H)\right)\cap L^p\left(\Omega\times(0,T); V\right)\cap C^0_w\left([0,T]; L^2(\Omega;H)\right)\,,
  \eeq
  \beq  \label{var_form_weak}
    \begin{split}
    X(t)-\int_0^t{\Div\eta(s)\,ds}+\int_0^t{\xi(s)\,ds}=&X_0+\int_0^t{B(s)\,dW_s} \quad\text{in } L^1(D)\cap V^*\,,\\
    & \Pi\text{-almost surely}\,, \quad\text{for every } t\in[0,T]\,.
    \end{split}
  \eeq
\end{thm}

\begin{rmk}
  The integral equation \eqref{var_form} is satisfied in the dual space $V_0^*$,
  but $X$ is not $V_0$-valued, so that the results provided are not
  a direct generalization of the classical concept of variational solution (cf.~\cite{prevot-rock}): 
  we can define them as a weaker type of variational
  solution, in which the integral expression holds in a dual space $V_0^*$, but the solution takes values only in a space larger than $V_0$
  ($V$ in our case). Nevertheless, the integral formulation
  \eqref{var_form} can be seen as an identity in $L^1(D)$, 
  so that the choice of $V_0$ turns out
  to be only a technical device {\em a posteriori}. 
  The fact that one cannot expect classical variational solutions for this type of problem is due to fact that
  no hypotheses on the growth of $\beta$ are assumed (in contrast to a large part of the literature).
\end{rmk}

\begin{rmk}
  Let us comment on hypothesis \eqref{ip3}. \luca{The fact that $\gamma$ is single-valued}
  (thus a continuous function) is needed in order to prove uniqueness for our problem, which in turn ensures
  some reasonable measurability properties for the processes $X$, $\eta$ and $\xi$, as we will show later on.
  On the other side, if we do not require \eqref{ip3}, the measurability of the solutions
  cannot be shown using the same argument,
  but it has to be recovered in a different way: however, in this case, the formulation that one obtains is weaker than the previous
  one, since the passage to the limit has to be carried out in $\Omega\times D$, with $t\in[0,T]$ fixed,
  and the solution $X$ is found is a larger space.
\end{rmk}

\begin{thm}
  \label{th2}
  In the setting \eqref{spaces}--\eqref{rel_V0}, assume that
  \begin{gather}
    \label{data_dep1}
    X_0^1, X_0^2 \in L^2\left(\Omega, \f_0, \Pi; H\right)\,,\\
    \label{data_dep2}
    \begin{split}
    B_1, B_2 \in L^2\left(\Omega\times(0,T); \cL_2(U,H)\right) \quad\text{progressively measurable}\,.
    \end{split}
  \end{gather}
  If hypothesis \eqref{ip3} holds and $(X_1, \eta_1, \xi_1)$, $(X_2, \eta_2, \xi_2)$ are any two corresponding
  solutions satisfying \eqref{sol1}--\eqref{j_integr}, then
  \beq
    \label{var_dep}
    \l|X_1-X_2\r|_{L^2\left(\Omega;L^\infty(0,T; H)\right)}
    \lesssim
    \l|X_0^1-X_0^2\r|_{L^2\left(\Omega; H\right)}+\l|B_1-B_2\r|_{L^2\left(\Omega\times(0,T);\cL_2(U,H)\right)}\,.
  \eeq
  In this setting, if $X_0^1=X_0^2$ and $B_1=B_2$, then $X_1=X_2$, $\eta_1=\eta_2$ and $\xi_1=\xi_2$.
  Moreover, if hypothesis \eqref{ip3} is not assumed and $(X_1, \eta_1, \xi_1)$, $(X_2, \eta_2, \xi_2)$ are any two corresponding solutions
  satisfying \eqref{sol2}--\eqref{sol3} and \eqref{incl_eta}--\eqref{var_form_weak},  then 
  \beq
  \label{weak_dep}
  \l|X_1-X_2\r|_{L^\infty\left(0,T;L^2(\Omega; H)\right)}
    \lesssim\l|X_0^1-X_0^2\r|_{L^2\left(\Omega; H\right)}+\l|B_1-B_2\r|_{L^2\left(\Omega\times(0,T);\cL_2(U,H)\right)}\,.
  \eeq
  In this setting, if $X_0^1=X_0^2$ and $B_1=B_2$, then $X_1=X_2$ and $-\Div\eta_1+\xi_1=-\Div\eta_2+\xi_2$.
\end{thm}

\begin{rmk}
  The uniqueness result strongly depends on the assumption \eqref{ip3}. Indeed, if \eqref{ip3} is in order,
  uniqueness holds for the three solution components, separately; on the other side, if we do not assume \eqref{ip3},
  we can only recover uniqueness for $X$ and the joint process $-\Div\eta+\xi$.
  Moreover, note that the nonlinearity $\gamma$ prevents us from finding a continuous dependence estimate
  also in the space $L^p(\Omega\times(0,T);V)$ for any $p>2$. Nevertheless, if 
  $p=2$ and $\gamma$ is the identity, the operator $-\Delta$ is linear
  and we can recover continuous dependence also in $L^2(\Omega\times(0,T);V)$, for which we refer to \cite{mar-scar}.
\end{rmk}

\begin{thm}
  \label{thm3}
  In the setting \eqref{spaces}--\eqref{rel_V0}, assume that
  \begin{gather}
    \label{ip1_mult}
    X_0\in L^2\left(\Omega,\f_0, \Pi; H\right)\,,\\
    \label{ip2_mult}
    B: \Omega\times[0,T]\times H\rarr \cL_2(U,H)\quad
    \text{ progressively measurable}\,,\\
    \label{ip3_mult}
    \begin{split}
    \exists\; L_B>0 \;: \quad&\l|B(\omega,t,x_1)-B(\omega, t, x_2)\r|_{\cL_2(U,H)}\leq L_B\l|x_1-x_2\r|_H\\
    &\text{for every}\quad (\omega, t)\in\Omega\times[0,T]\,, \;\;x_1, x_2\in H\,,
    \end{split}\\
   \label{ip4_mult}
   \begin{split}
    \exists\; R_B>0 \;: \quad&\l|B(\omega,t,x)\r|_{\cL_2(U,H)}\leq R_B\left(1+\l|x\r|_H\right)
    \quad\forall\; (\omega, t, x)\in\Omega\times[0,T]\times H\,.
    \end{split}
  \end{gather}
  If hypothesis \eqref{ip3} holds, then there exists a triplet $(X,\eta,\xi)$ satisfying conditions \eqref{sol1}--\eqref{sol3}, \eqref{incl_eta}--\eqref{j_integr} and
  \beq
    \label{var_form_mult}
    \begin{split}
    X(t)-\int_0^t{\Div\eta(s)\,ds}+\int_0^t{\xi(s)\,ds}=&X_0+\int_0^t{B(s, X(s))\,dW_s} \quad\text{in } L^1(D)\cap V^*\,,\\
    &\text{for every } t\in[0,T]\,, \quad\Pi\text{-almost surely}\,.
    \end{split}
    \eeq
    If hypothesis \eqref{ip3} is not assumed, then the same conclusion
    is true replacing \eqref{sol1} with \eqref{sol1_weak},
    and condition \eqref{var_form_mult} with
  \beq
  \label{var_form_mult_weak}
    \begin{split}
    X(t)-\int_0^t{\Div\eta(s)\,ds}+\int_0^t{\xi(s)\,ds}=&X_0+\int_0^t{B(s, X(s))\,dW_s} \quad\text{in } L^1(D)\cap V^*\,,\\
    &\Pi\text{-almost surely}\,, \quad\text{for every } t\in[0,T]\,.
    \end{split}
  \eeq
\end{thm}
\begin{thm}
  \label{thm4}
  In the setting \eqref{spaces}--\eqref{rel_V0}, let $X_0^1, X_0^2$ satisfy condition \eqref{data_dep1}.
  If hypothesis \eqref{ip3} holds and $(X_1, \eta_1, \xi_1)$, $(X_2, \eta_2, \xi_2)$ are any two corresponding solutions satisfying 
  \eqref{sol1}--\eqref{sol3}, \eqref{incl_eta}--\eqref{j_integr} and \eqref{var_form_mult}, then
  \beq
    \label{var_dep_mult}
    \l|X_1-X_2\r|_{L^2\left(\Omega;L^\infty(0,T; H)\right)}
    \lesssim
    \l|X_0^1-X_0^2\r|_{L^2\left(\Omega; H\right)}\,.
  \eeq
  In this setting, if $X_0^1=X_0^2$, then $X_1=X_2$, $\eta_1=\eta_2$ and $\xi_1=\xi_2$.
  Moreover, if hypothesis \eqref{ip3} is not assumed and $(X_1, \eta_1, \xi_1)$, $(X_2, \eta_2, \xi_2)$ are any two corresponding solutions
  satisfying \eqref{sol2}--\eqref{sol3}, \eqref{incl_eta}--\eqref{sol1_weak} and \eqref{var_form_mult_weak},  then 
  \beq
  \label{var_dep_mult_weak}
  \l|X_1-X_2\r|_{L^\infty\left(0,T;L^2(\Omega; H)\right)}
    \lesssim\l|X_0^1-X_0^2\r|_{L^2\left(\Omega; H\right)}\,.
  \eeq
  In this setting, if $X_0^1=X_0^2$, then $X_1=X_2$ and $-\Div\eta_1+\xi_1=-\Div\eta_2+\xi_2$.
\end{thm}

\begin{rmk}
It is worth recalling the classical approach to problem \eqref{eq}--\eqref{bound} in the deterministic case
and the main differences with the stochastic case.
The corresponding deterministic problem is
\[
  \frac{\partial u}{\partial t}-\Div\gamma(\nabla u)+\beta(u)\ni f\,, \quad u(0)=u_0\,,
\]
with homogeneous boundary conditions for $u$: here, the classical approach consists in proving 
that the sum of the two operators $-\Div(\nabla\cdot)$ and $\beta(\cdot)$ is $m$-accretive in a suitable space.
To this end, it is well-known that if $(i)$ $E$ is a Banach space with uniformly convex dual $E^*$, $(ii)$
$A$ and $B$ are two $m$-accretive sets in $E\times E$, $(iii)$ $D(A)\cap D(B)\neq\emptyset$,
$(iv)$ $\left<Au, J(B_\lambda u)\right>_{E}\geq0$ for every $u\in D(A)$ and $\lambda\in(0,1)$
(where $J:E\rarr E^*$ is the duality mapping of $E$ and $B_\lambda$ is the Yosida approximation of $B$),
then $A+B$ is $m$-accretive in $E\times E$ (see \cite[Prop.~3.8]{barbu_monot}).
If we take for example $E=L^s(D)$ for $1<s<+\infty$, $A=-\Div\gamma(\nabla\cdot)$,
$B=\beta(\cdot)$ with their natural domains,
we only need to check $(iv)$, since $(i)$--$(iii)$ are clearly satisfied. To this aim, we need to handle the term
\[
\int_D{-\Div\gamma(\nabla u)\phi(\beta_\lambda(u))}\,,
\]
where $\phi(r)=|r|^{s-2}r$, $r\in\Ar$, using integration by parts. The first problem occurs if $s<2$,
since in this case the derivative of $\phi$ explodes at $0$;
if $s\geq2$, we can proceed formally and recover
\[
\int_D\phi'(\beta_\lambda(u))\beta'_\lambda(u)\gamma(\nabla u)\cdot\nabla u\geq0\,.
\]
The main difficulty is that $\beta_\lambda$ is not differentiable, so that one needs to rely on some
generalized chain-rules for Lipschitz functions or suitable mollifications of $\beta_\lambda$.
The problem can be seen then as a particular case of the general one
\[
  \frac{\partial u}{\partial t}+ Au \ni f\,,
\]
with $A$ purely nonlinear (multivalued) operator, for which one can rely on several classical well-posedness results.
However, the corresponding general problem in the stochastic case, i.e.
\[
du+Au\,dt\ni B\,dW_t\,,
\]
does not have a 
direct counterpart in terms of existence and uniqueness:
as a consequence, in our case the proof of $m$-accretivity
is not sufficient to ensure well-posedness, so that one needs to deal with the problem "by hand".
To this end, the variational approach is in order.
\end{rmk}


\section{Existence with additive noise}
\setcounter{equation}{0}
\label{exist}

In this section we prove the two existence results contained in Theorem \ref{th1}:
as already mentioned, we are going to approximate the problem
using two different parameters. Uniform estimates are then proved and
we obtain 
global solutions to the original problem
by passing to the limit in a suitable topology.

\subsection{The approximated problem}

Thanks to \eqref{ip2}, for every $\eps\in(0,1)$ there exists an operator
\beq
  \label{B_eps}
  B^\eps\in L^2\left(\Omega\times(0,T);\cL_2(U,V_0)\right)
\eeq
such that:
\beq
  \label{B_eps1}
  B^\eps\rarr B 
  \quad\text{in } L^2\left(\Omega\times(0,T);\cL_2(U,H)\right) \quad\text{as } \eps\searrow0\,,
\eeq
\beq
  \label{B_eps2}
  \l|B^\eps\r|_{L^2\left(\Omega\times(0,T);\cL_2(U,H)\right)}\leq\l|B\r|_{L^2\left(\Omega\times(0,T);\cL_2(U,H)\right)}
  \quad\text{for every } \eps\in(0,1)\,.
\eeq
Indeed, if $k$ is chosen as in \eqref{V0},
then
the operator $(I-\eps\Delta)^{-k}$ maps $H$ into $V_0$ for every $\eps>0$, so that it suffices to take
$B^\eps:=(I-\eps\Delta)^{-k}B$. With this particular choice, using the fact that
the operator $(I-\eps\Delta)^{-k}:H\rarr H$ is a linear contraction converging to the identity
in the strong operator topology as $\eps\searrow0$ and the ideal property of
$\cL_2(U;H)$ in $\cL(U,H)$, we have that
\eqref{B_eps}--\eqref{B_eps2} are satisfied.

For every $\lambda\in(0,1)$ and $\eps\in(0,1)$, let us consider the approximated problem
\begin{gather}
  dX^\eps_\lambda-\Div[\gamma_{\lambda}(\nabla X^\eps_\lambda)+\lambda\nabla X^\eps_\lambda]\,dt
  +\beta_\lambda(X^\eps_\lambda)\,dt= B^\eps\,dW_t \quad\text{in } D\times(0,T)\,,\\
  X^\eps_\lambda(0)=X_0 \quad\text{in } D\,,
\end{gather}
whose
integral formulation is given by
\beq
  \label{eq'}
  \begin{split}
  &X^\eps_\lambda(t)-\int_0^t{\Div[\gamma_{\lambda}(\nabla X^\eps_\lambda(s))]\,ds}-\lambda\int_0^t{\Delta X^\eps_\lambda(s)\,ds}
  +\int_0^t{\beta_\lambda(X^\eps_\lambda(s))\,ds}\\
  &\quad=X_0+\int_0^t{B^\eps(s)\,dW_s} \quad\text{in $H^{-1}(D)$}\,,
  \quad\text{for every } t\in[0,T]\,, \quad\Pi\text{-almost surely}\,,
  \end{split}
\eeq
where here $-\Div: L^2(D)^d\rarr H^{-1}(D)$ and the laplacian is intended in the usual variational way, i.e.
\[
-\Delta:H^1_0(D)\rarr H^{-1}(D)\,, \quad \left<-\Delta u, v\right>_{H^1_0(D)}:=\int_D{\nabla u\cdot \nabla v}\,, \quad u,v\in H^1_0(D)\,.
\]

A unique solution to the approximated problem \eqref{eq'} can be
easily obtained using the classical results contained 
in \cite{KR-spde} (see also \cite[Thm.~4.2.4]{prevot-rock}). 
In fact, the operator
\beq
  \label{A_lambda}
  A_\lambda:H^1_0(D)\rarr H^{-1}(D)\,, \quad 
  A_\lambda:\phi\mapsto-\Div[\gamma_{\lambda}(\nabla\phi)+\lambda\nabla\phi]+\beta_\lambda(\phi)\,,
\eeq
is well-defined thanks to the Lipschitz continuity of $\beta_\lambda$ and $\gamma_\lambda$, and
problem \eqref{eq'} is the variational formulation with respect to
the Gelfand triple $H^1_0(D)\embed H\embed H^{-1}(D)$ of the following:
\begin{gather}
  \label{eq''}
  dX^\eps_\lambda+A_\lambda X^\eps_\lambda\,dt=B^\eps\,dW_t \quad\text{in } (0,T)\times D\,,\\
  \label{eq''_bis}
  X^\eps_\lambda(0)=X_0 \quad\text{in } D\,.
\end{gather}
In this setting, we need to check that the operator $A_\lambda$ satisfies the classical properties
of hemicontinuity, monotonicity, coercivity and boundedness, in order to recover solutions of \eqref{eq'}.
The following lemma is straightforward.

\begin{lem}
  \label{properties}
  The following conditions are satisfied for every $\lambda\in(0,1)$.
  \begin{itemize}
    \item[(H1)] (Hemicontinuity). For all $u,v,x\in H^1_0(D)$, the following map is continuous:
        \[
          s\mapsto\left<A_\lambda(u+s v),x\right>_{H^1_0(D)}\,, \quad s\in\Ar\,.
        \]
    \item[(H2)] (Monotonicity). For all $u,v\in H^1_0(D)$,
      \[
        \left<A_\lambda u-A_\lambda v, u-v\right>_{\gH1}\geq0\,.
      \]
    \item[(H3)] (Coercivity). There exists $C_1>0$ such that, for all $v\in \gH1$,
      \[
        \left<A_\lambda v, v\right>_{\gH1} \geq C_1\l|v\r|_{\gH1}^2\,.
      \]
     \item[(H4)] (Boundedness). There exists $C_2>0$ such that, for all $v\in \gH1$,
       \[
         \l|A_\lambda v\r|_{H^{-1}(D)}\leq C_2\l|v\r|_{\gH1}\,.
       \]
  \end{itemize}
\end{lem}
\begin{proof}
  For all $u,v,x\in \gH1$ we have
  \[
    \begin{split}
    \left<A_\lambda(u+s v),x\right>_{\gH1}&=
    \int_D{\gamma_{\lambda}(\nabla(u+s v))\cdot\nabla x}
    +\lambda\int_D{\nabla (u+s v)\cdot\nabla x}
    +\int_D{\beta_\lambda(u+s v)x}\,,
    \end{split}
  \]
  so that (H1) is satisfied thanks to the Lipshitz continuity of $\gamma_{\lambda}$ and $\beta_\lambda$.
  Secondly, (H2) trivially holds using the monotonicity of $\gamma_{\lambda}$ and $\beta_\lambda$.
  Moreover, for all $v\in \gH1$, thanks to the monotonicity of $\gamma_\lambda$ and $\beta_\lambda$,
  and the fact that $\gamma(0)\ni0$ and $\beta(0)\ni0$, we have
  \[ 
    \left<A_\lambda v, v\right>_{\gH1}=
    \int_D{\gamma_{\lambda}(\nabla v)\cdot\nabla v}+\lambda\int_D{\left|\nabla v\right|^2}+\int_D{\beta_\lambda(v)v}\geq
    \lambda\int_D{\left|\nabla v\right|^2}\,,
  \]
  so that (H3) holds true thanks to the Poincaré inequality.
  Finally, using the Lipschitz continuity of $\beta_\lambda$ and $\gamma_\lambda$ and 
  the Hölder inequality,
  we have for all $u,v\in \gH1$
  \[
    \begin{split}
   \left<A_\lambda v,u\right>_{\gH1}&=
   \int_D{\gamma_{\lambda}(\nabla v)\cdot\nabla u}+\lambda\int_D{\nabla v\cdot\nabla u}+\int_D{\beta_\lambda(v)u}\\
   &\leq
   \left(\frac{1}{\lambda}+\lambda\right)\l|\nabla v\r|_H\l|\nabla u\r|_H+\frac{1}{\lambda}\l|v\r|_H\l|u\r|_H
   \leq
   \left(\frac{2}{\lambda}+\lambda\right)\l|v\r|_{\gH1}\l|u\r|_{\gH1}\,,
    \end{split}
  \]
  from which 
  (H4) follows.
\end{proof}

Lemma \ref{properties} ensures that, for all $\eps, \lambda\in(0,1)$, there exists a unique adapted process
\beq
  \label{sol_app}
  X^\eps_\lambda\in
  L^2\left(\Omega; C^0\left([0,T];H\right)\right)\cap
  L^2\left(\Omega\times(0,T);H^1_0(D)\right)
\eeq
such that
\beq
  \label{app1}
  \begin{split}
  &X^\eps_\lambda(t)-\int_0^t{\Div[\gamma_{\lambda}(\nabla X^\eps_\lambda(s))]\,ds}
  -\lambda\int_0^t{\Delta X^\eps_\lambda(s)\,ds}+\int_0^t{\beta_\lambda(X^\eps_\lambda(s))\,ds}\\
  &\quad=X_0+\int_0^t{B^\eps(s)\,dW_s} \quad\text{in $H^{-1}(D)$}\,,
  \quad\text{for every } t\in[0,T]\,, \quad\Pi\text{-almost surely}\,.
  \end{split}
\eeq

\subsection{A priori estimates I}
\label{first}

Here we prove uniform pathwise estimates on $X^\eps_\lambda$, independent of $\lambda$ (but not of $\eps$),
which will allow us to pass to the limit as $\lambda\searrow0$
in the approximated problem \eqref{app1} with $\eps$ fixed.

Let us define, for any $\eps\in(0,1)$,
\beq
  \label{stoc}
  W^\eps_B(t):=\int_0^t{B^\eps(s)\,dW_s}\,, \quad t\in[0,T]\,.
\eeq
Thanks to the Burkholder-Davis-Gundy inequality and condition \eqref{ip2} we deduce
\beq
  \label{stoc_reg}
  W^\eps_B\in L^2\left(\Omega; L^\infty(0,T;V_0)\right)\,.
\eeq
In particular, recalling \eqref{rel_V0}, we have that
\beq
  \label{stoc_reg'}
  W_B^\eps(\omega)\in L^p(0,T;V)\cap L^\infty(Q) \quad\text{for $\Pi$-almost every } \omega\in\Omega\,.
\eeq

Equation \eqref{app1} can be rewritten as
\[
  \partial_t\left(X^\eps_\lambda-W^\eps_B\right)(t)-
  \Div\left[\gamma_{\lambda}(\nabla X^\eps_\lambda(t))+\lambda\nabla X^\eps_\lambda(t)\right]
  +\beta_\lambda(X^\eps_\lambda(t))=0 \quad\text{in } H^{-1}(D)
\]
for every $t\in[0,T]$, for any $\omega$ out of a set of probability $0$ (the symbol $\partial_t$ for the derivative with respect to time makes sense only if applied to the
difference $X^\eps_\lambda-W^\eps_B$). Fix now $\omega$ and
test by $X^\eps_\lambda(t)-W^\eps_B(t)$ (see \cite[{\S}1.3]{barbu-syst}): we obtain 
\beq
  \label{test_lam_eps}
  \begin{split}
  \frac{1}{2}\l|X^\eps_\lambda(t)-W^\eps_B(t)\r|^2_H&+
  \int_0^t\int_D{\gamma_{\lambda}(\nabla X^\eps_\lambda(s))\cdot\nabla(X^\eps_\lambda(s)-W^\eps_B(s))\,ds}\\
  &+\lambda\int_0^t\int_D{\nabla X^\eps_\lambda(s)\cdot\nabla\left(X^\eps_\lambda(s)-W^\eps_B(s)\right)\,ds}\\
  &+\int_0^t{\int_D\beta_\lambda(X^\eps_\lambda(s))(X^\eps_\lambda(s)-W^\eps_B(s))\,ds}=\frac{1}{2}\l|X_0\r|^2_H\,.
  \end{split}
\eeq
Using the identity $I_d=\lambda\gamma_\lambda+J_\lambda$ and rearranging terms in the previous relation, we have
\[
  \begin{split}
  \frac{1}{2}&\l|X^\eps_\lambda(t)-W^\eps_B(t)\r|^2_H+
  \int_0^t\int_D{\gamma_{\lambda}(\nabla X^\eps_\lambda(s))\cdot J_\lambda\left(\nabla X^\eps_\lambda(s)\right)\,ds}
  +\lambda\int_0^t\int_D{\left|\gamma_\lambda\left(\nabla X^\eps_\lambda(s)\right)\right|^2\,ds}\\
  &\qquad\qquad\qquad+\lambda\int_0^t\int_D{|\nabla X^\eps_\lambda(s)|^2\,ds}
  +\int_0^t{\int_D\beta_\lambda(X^\eps_\lambda(s))(X^\eps_\lambda(s)-W^\eps_B(s))\,ds}\\
  &=\frac{1}{2}\l|X_0\r|^2_H+\int_0^t\int_D{\gamma_\lambda\left(\nabla X^\eps_\lambda(s)\right)\cdot\nabla W^\eps_B(s)\,ds}
  +\lambda\int_0^t\int_D{\nabla X^\eps_\lambda(s)\cdot\nabla W^\eps_B(s)\,ds}\,.
  \end{split}
\]
Using the generalized Young inequality
of the form $ab\leq\delta\frac{p-1}{p}a^{\frac{p}{p-1}}+C_{\delta,p} b^p$
(for any $a,b,\delta>0$ and a certain $C_{\delta, p}>0$)
on the second term on the right-hand side,
thanks also to hypotheses \eqref{gamma1}--\eqref{gamma2} and condition \eqref{prop_res2} we deduce
for every $t\in[0,T]$ that
\[
  \begin{split}
  \frac{1}{2}&\l|X^\eps_\lambda(t)-W^\eps_B(t)\r|^2_H+
  K\int_0^t{\l|J_\lambda\left(\nabla X^\eps_\lambda(s)\right)\r|^p_{L^p(D)}\,ds}
  +\lambda\int_0^t{\l|\gamma_\lambda\left(\nabla X^\eps_\lambda(s)\right)\r|^2_{H}\,ds}\\
  &\qquad\qquad\qquad+\lambda\int_0^t{\l|\nabla X^\eps_\lambda(s)\r|_H^2\,ds}
  +\int_0^t{\int_D\beta_\lambda(X^\eps_\lambda(s))(X^\eps_\lambda(s)-W^\eps_B(s))\,ds}\\
  &\leq C'+\frac{1}{2}\l|X_0\r|^2_H+
  \delta\frac{(p-1)D_1}{p}\int_0^t{\l|J_\lambda\left(\nabla X^\eps_\lambda(s)\right)\r|_{L^p(D)}^p\,ds}+
  C_{\delta, p}\int_0^t{\l|\nabla W^\eps_B(s)\r|^p_{L^p(D)}\,ds}\\
  &\qquad\qquad\qquad+\frac{\lambda}{2}\int_0^t{\l|\nabla X^\eps_\lambda(s)\r|^2_H\,ds}
  +\frac{\lambda}{2}\int_0^t{\l|\nabla W^\eps_B(s)\r|^2_H\,ds}
  \end{split}
\]
for a positive constants $C'$ independent of $\lambda$ and $\eps$.
Hence, choosing $\delta=\frac{Kp}{2D_1(p-1)}$, we get that, for every $t\in[0,T]$,
\beq
  \label{int_ex}
  \begin{split}
  &\l|X^\eps_\lambda(t)\r|^2_H+\frac{K}{2}\int_0^t{\l|J_\lambda\left(\nabla X^\eps_\lambda(s)\right)\r|_{L^p(D)}^p\,ds}
  +\lambda\int_0^t{\l|\gamma_\lambda\left(\nabla X^\eps_\lambda(s)\right)\r|^2_{H}\,ds}\\
  &\qquad\qquad\qquad+\frac{\lambda}{2}\int_0^t{\l|\nabla X^\eps_\lambda(s)\r|^2_H\,ds}
  +\int_0^t{\int_D\beta_\lambda(X^\eps_\lambda(s))(X^\eps_\lambda(s)-W^\eps_B(s))\,ds}\\
  &\quad\leq C'+\frac{1}{2}\l|X_0\r|^2_H+C_p\l|W^\eps_B\r|_{L^p(0,T;V)}^p+\frac{1}{2}\l|W^\eps_B\r|^2_{L^\infty(0,T;H)}
  +\frac{1}{2}\l|W_B^\eps\r|^2_{L^2(0,T;H^1_0(D))}
  \end{split}
\eeq
for a positive constant $C_p$ independent of $\lambda$ and $\eps$.
Denoting by $j_\lambda:\Ar\rarr[0,+\infty)$ the proper, convex, lower semicontinuous function such that 
$\beta_\lambda=\partial j_\lambda$ and $j_\lambda(0)=0$, one has that $j_\lambda\leq j$
and $j_\lambda(x)\nearrow j(x)$ for every $x\in\Ar$ (recall that $\Ar=D(\beta)\subseteq D(j)$).
Hence, for every $x,y\in\Ar$ we have that
\[
  \beta_\lambda(x)(x-y)\geq j_\lambda(x)-j_\lambda(y)\geq j_\lambda(x)-j(y)\,.
\]
Applying this inequality to the last term on the left-hand side of \eqref{int_ex}, we deduce that, for every $t\in[0,T]$,
\[
  \begin{split}
  &\l|X^\eps_\lambda(t)\r|^2_H+\frac{K}{2}\int_0^t{\l|J_\lambda\left(\nabla X^\eps_\lambda(s)\right)\r|_{L^p(D)}^p\,ds}
  +\lambda\int_0^t{\l|\gamma_\lambda\left(\nabla X^\eps_\lambda(s)\right)\r|^2_{H}\,ds}\\
  &\qquad\qquad\qquad+\frac{\lambda}{2}\int_0^t{\l|\nabla X^\eps_\lambda(s)\r|^2_H\,ds}
  +\int_0^t{\int_D j_\lambda(X^\eps_\lambda(s))\,ds}\\
  &\quad\lesssim 1+\l|X_0\r|^2_H+\l|W^\eps_B\r|^2_{L^p(0,T;V)}+\l|W^\eps_B\r|^2_{L^\infty(0,T;H)}
  +\l|W_B^\eps\r|^2_{L^2(0,T;H^1_0(D))}+\int_Q{j(W^\eps_B)}\,.
  \end{split}
\]
Note that all the terms on the right-hand side are finite $\Pi$-almost surely: for the first five, this is immediate
thanks to \eqref{ip1} and \eqref{stoc_reg'}, while
$j(W^\eps_B)\in L^1(Q)$ since $W^\eps_B\in L^\infty(Q)$.
Using the positivity of $j_\lambda$
we deduce that
for $\Pi$-almost every $\omega\in\Omega$
there exists a positive constant $M=M_{\omega,\eps}$, independent of $\lambda$,
such that, for every $\lambda\in(0,1)$,
\begin{gather}
  \label{est1}
  \l|X^\eps_\lambda(\omega)\r|_{L^\infty(0,T;H)}\leq M_{\omega,\eps}\,,\\
  \label{est_res}
  \l|J_\lambda\left(\nabla X^\eps_\lambda(\omega)\right)\r|_{L^p(Q)}\leq M_{\omega,\eps}\,,\\
  \label{est_gamma_lam}
  \lambda^{1/2}\l|\gamma_\lambda\left(\nabla X^\eps_\lambda(\omega)\right)\r|_{L^2(Q)}
  \leq M_{\omega, \eps}\,,\\
  \label{est_grad}
  \lambda^{1/2}\l|\nabla X^\eps_\lambda(\omega)\r|_{L^2(Q)}\leq M_{\omega, \eps}\,.
\end{gather}
Finally, by \eqref{gamma2} and \eqref{prop_res2} we also have
\[
  \int_{Q}{\left|\gamma_\lambda(\nabla X^\eps_\lambda)\right|^q}\leq
  D_1\int_{Q}{\left(1+|J_\lambda(\nabla X^\eps_\lambda)|\right)^p}\,,
\]
so that by \eqref{est_res} it follows (possibly redefining $M_{\omega, \eps}$) that, for every $\lambda\in(0,1)$,
\beq
  \label{est_gamma}
  \l|\gamma_\lambda(\nabla X^\eps_\lambda(\omega))\r|_{L^q(Q)}\leq M_{\omega, \eps}\,.
\eeq

\subsection{A priori estimates II}
\label{second}

In this section we prove some estimates in expectation on $X^\eps_\lambda$ independent both of $\lambda$ and $\eps$.
The main tool is a version of It\^o's formula in a 
variational framework.

Thanks to conditions \eqref{ip1}--\eqref{ip2} and \eqref{sol_app}--\eqref{app1}, 
we can apply It\^o's formula (see \cite[Thm.~4.2.5]{prevot-rock}), obtaining 
\beq
  \label{ito}
  \begin{split}
  \frac{1}{2}\l|X^\eps_\lambda(t)\r|^2_H&+\int_0^t\int_D{\gamma_\lambda(\nabla X^\eps_\lambda(s))\cdot\nabla X_\lambda^\eps(s)\,ds}
  +\lambda\int_0^t\int_D{\left|\nabla X^\eps_\lambda(s)\right|^2\,ds}\\
  &\qquad\qquad\qquad+\int_0^t\int_D{\beta_\lambda(X^\eps_\lambda(s))X^\eps_\lambda(s)\,ds}\\
  &=\frac{1}{2}\l|X_0\r|_H^2+
  \frac{1}{2}\int_0^t{\l|B^\eps(s)\r|_{\cL_2(U,H)}^2\,ds}
  +\int_0^t\left(X^\eps_\lambda(s), B^\eps(s)\,dW_s\right)
  \end{split}
\eeq
for every $t\in[0,T]$, $\Pi$-almost surely, which yields, by definition of $\gamma_\lambda$
and conditions \eqref{gamma2} and \eqref{prop_res2},
\[
  \begin{split}
  \frac{1}{2}\l|X^\eps_\lambda(t)\r|^2_H&+K\int_0^t{\l|J_\lambda\left(\nabla X^\eps_\lambda(s)\right)\r|^p_{L^p(D)}\,ds}
  +\lambda\int_0^t{\l|\gamma_\lambda\left(\nabla X^\eps_\lambda(s)\right)\r|^2_{H}\,ds}\\
  &\qquad\qquad\qquad+\lambda\int_0^t{\l|\nabla X^\eps_\lambda(s)\r|^2_H\,ds}
  +\int_0^t\int_D{\beta_\lambda(X^\eps_\lambda(s))X^\eps_\lambda(s)\,ds}\\
  &\leq C''+\frac{1}{2}\l|X_0\r|_H^2+
  \frac{1}{2}\l|B^\eps(s)\r|_{L^2(0,T; \cL_2(U,H))}^2
  +\sup_{t\in[0,T]}\left|\int_0^t\left(X^\eps_\lambda(s), B^\eps(s)\,dW_s\right)\right|
  \end{split}
\] 
for a constant $C''>0$, independent of $\eps$ and $\lambda$.
Thanks to Davis' inequality, the Hölder and Young inequalities, and condition \eqref{B_eps2}, we have
\[
  \begin{split}
  &\exval\sup_{t\in[0,T]}\left|\int_0^t\left(X^\eps_\lambda(s), B^\eps(s)\,dW_s\right)\right|
  \lesssim\exval\left[\left(\int_0^T\l|X^\eps_\lambda(s)\r|^2_H\l|B^\eps(s)\r|^2_{\cL_2(U,H)}\,ds\right)^{1/2}\right]\\
  &\leq\exval\left[\l|X^\eps_\lambda\r|_{L^\infty(0,T;H)}\l|B^\eps\r|_{L^2(0,T;\cL_2(U,H))}\right]
  \leq\frac{1}{4}\l|X^\eps_\lambda\r|^2_{L^2(\Omega; L^\infty(0,T;H))}+\l|B\r|^2_{L^2(\Omega\times(0,T); \cL_2(U,H))}\,:
  \end{split}
\]
consequently, taking the supremum in $t\in[0,T]$ and then expectations, we obtain
\beq
  \label{ito'}
  \begin{split}
  \frac{1}{4}\l|X^\eps_\lambda\r|^2_{L^2(\Omega; L^\infty(0,T;H))}&+
  K\l|J_\lambda\left(\nabla X^\eps_\lambda\right)\r|^p_{L^p(\Omega\times(0,T)\times D)}+
  \lambda\l|\gamma_\lambda\left(\nabla X^\eps_\lambda\right)\r|^2_{L^2(\Omega\times(0,T)\times D)}\\
  &+\lambda\l|\nabla X^\eps_\lambda\r|^2_{L^2(\Omega\times(0,T)\times D)}+
  \int_{\Omega\times Q}{\beta_\lambda(X^\eps_\lambda)X_\lambda^\eps}\\
  &\leq C''+\frac{1}{2}\l|X_0\r|^2_{L^2(\Omega; H)}+\frac{3}{2}\l|B\r|^2_{L^2(\Omega\times (0,T);\cL_2(U,H))}\,.
  \end{split}
\eeq
We infer that there exists a constant $N>0$, independent of $\lambda$ and $\eps$, such that
\begin{gather}
  \label{est1'}
  \l|X^\eps_\lambda\r|_{L^2(\Omega; L^\infty(0,T; H))}\leq N\,,\\
  \label{est_res'}
  \l|J_\lambda\left(\nabla X^\eps_\lambda\right)\r|_{L^p(\Omega\times(0,T)\times D)}\leq N\,,\\
  \label{est_gamma_lam'}
  \lambda^{1/2}\l|\gamma_\lambda\left(\nabla X^\eps_\lambda\right)\r|_{L^2(\Omega\times(0,T)\times D)}
  \leq N\,,\\
  \label{est_grad'}
  \lambda^{1/2}\l|\nabla X^\eps_\lambda\r|_{L^2(\Omega\times(0,T)\times D)}\leq N\,,
\end{gather}
for every $\eps,\lambda\in(0,1)$.
Finally, by \eqref{gamma2} and \eqref{prop_res2} we also have
\[
  \int_{\Omega\times Q}{\left|\gamma_\lambda(\nabla X^\eps_\lambda)\right|^q}\leq
  D_1\int_{\Omega\times Q}{\left(1+|J_\lambda(\nabla X^\eps_\lambda)|\right)^p}\,,
\]
so that by \eqref{est_res'} it follows (possibly redefining $N$) that, for every $\eps,\lambda\in(0,1)$,
\beq
  \label{est_gamma'}
  \l|\gamma_\lambda(\nabla X^\eps_\lambda)\r|_{L^q(\Omega\times(0,T)\times D)}\leq N\,.
\eeq

\subsection{A priori estimates III}
\label{third}

In this section we prove uniform estimates on the term $\beta_\lambda(X^\eps_\lambda)$, independent of $\lambda$
(with $\eps$ fixed), which are useful to recover a suitable weak compactness.
We rely on some computations performed in \cite{barbu} 
to obtain some $L^1$ estimates, the classical results by de la Vallée-Poussin about uniform integrability and on
the Dunford-Pettis theorem.

Firstly, let us fix $\omega\in\Omega$. 
Property \eqref{convex2}, conditions \eqref{prop_res1}--\eqref{prop_res2} and
the monotonicity of $\beta_\lambda$ imply that
\[
  j(R_\lambda X^\eps_\lambda)+j^*(\beta_\lambda(X^\eps_\lambda))=
  \beta_\lambda(X^\eps_\lambda) R_\lambda X^\eps_\lambda \leq \left|\beta_\lambda(X^\eps_\lambda)\right|\left| X^\eps_\lambda\right|=
  \beta_\lambda(X^\eps_\lambda) X^\eps_\lambda\,.
\]
Consequently, from inequality \eqref{int_ex} evaluated at time $T$ and the previous relation, recalling \eqref{stoc_reg'}
and using the generalized Young inequality of the form $ab\leq j(2a)+j^*(b/2)$ for any $a,b\in\Ar$ (see \eqref{convex3}),
we deduce that
$\Pi$-almost surely we have
\[
  \begin{split}
  \int_Q{j^*(\beta_\lambda(X^\eps_\lambda))}&\leq\int_Q{\beta_\lambda(X^\eps_\lambda)X^\eps_\lambda}
  \leq C'+ \frac{1}{2}\l|X_0\r|^2_H+C_p\l|W^\eps_B\r|_{L^p(0,T;V)}^p\\
  &\qquad\qquad+\frac{1}{2}\l|W^\eps_B\r|^2_{L^\infty(0,T;H)}+
  \frac{1}{2}\l|W^\eps_B\r|_{L^2(0,T; H^1_0(D))}^2
  +\int_Q{\beta_\lambda(X^\eps_\lambda)W^\eps_B}\\
  &\leq C'+ \frac{1}{2}\l|X_0\r|^2_H+C_p\l|W^\eps_B\r|_{L^p(0,T;V)}^p
  +\frac{1}{2}\l|W^\eps_B\r|^2_{L^\infty(0,T;H)}\\
  &\qquad\qquad+\frac{1}{2}\l|W^\eps_B\r|_{L^2(0,T; H^1_0(D))}^2
  +\l|j\left(2W^\eps_B\right)\r|_{L^1(Q)}+\frac{1}{2}\int_Q{j^*\left(\beta_\lambda(X^\eps_\lambda)\right)}\,.
  \end{split}
\]
All the terms on the right hand side are finite thanks to \eqref{ip1} and \eqref{stoc_reg'}:
hence,
since $j^*$ is even by assumption,
we have proved that 
\beq
    \label{est_j_star}
    \l|j^*\left(|\beta_\lambda(X^\eps_\lambda(\omega))|\right)\r|_{L^1(Q)}=
     \l|j^*\left(\beta_\lambda(X^\eps_\lambda(\omega))\right)\r|_{L^1(Q)}\leq
     \int_Q{\beta_\lambda(X^\eps_\lambda(\omega))X^\eps_\lambda(\omega)}\leq M_{\omega,\eps}
\eeq
for $\Pi$-almost every $\omega\in\Omega$;
moreover, since $D(\beta)=\Ar$ by \eqref{beta}, 
we have that
\[
  \lim_{|r|\rarr+\infty}{\frac{j^*(r)}{|r|}}=+\infty
\]
(see for example \cite{barbu_monot, brezis}).
Hence, using then the criterion by de la Vallée-Poussin for uniform integrability
combined with the Dunford-Pettis theorem,
we deduce that,
for $\Pi$-almost every $\omega\in\Omega$ and for every $\eps\in(0,1)$,
\beq
  \label{comp_beta}
  \left\{\beta_{\lambda}(X^\eps_\lambda)(\omega)\right\}_{\lambda\in(0,1)} \quad\text{is weakly relatively compact in } L^1\left(Q\right)\,.
\eeq

Finally, let us obtain the corresponding information also in expectation.
It easily follows from \eqref{ito'} that there exists a constant $N>0$, independent of $\lambda$ and $\eps$, such that
\[
  \l|\beta_\lambda(X^\eps_\lambda)X^\eps_\lambda\r|_{L^1(\Omega\times(0,T)\times D)}\leq N \quad\text{for every }
  \eps, \lambda\in(0,1)\,;
\]
hence, in analogy to the derivation of \eqref{est_j_star}, we get
\beq
  \label{est_j_star'}
  \int_{\Omega\times Q}{j^*(\beta_\lambda(X^\eps_\lambda))}\leq
  \l|\beta_\lambda(X^\eps_\lambda)X^\eps_\lambda\r|_{L^1(\Omega\times(0,T)\times D)}\leq N \quad\text{for every } \eps, \lambda\in(0,1)\,.
\eeq
Since $j^*$ is even and superlinear at infinity, the criterion by de la Vallée-Poussin and the Dunford-Pettis theorem imply that
\beq
  \label{comp_beta'}
  \{\beta_\lambda(X^\eps_\lambda)\}_{\eps, \lambda\in(0,1)} \quad\text{is weakly relatively compact in } L^1(\Omega\times(0,T)\times D)\,.
\eeq

\subsection{Passage to the limit as $\lambda\searrow0$}
\label{first_lim}

In this section, we pass to the limit as $\lambda\searrow0$ in the approximated problem \eqref{app1} with 
$\eps\in(0,1)$ being fixed: the idea is to pass to the limit pathwise as $\lambda\searrow0$. 
Throughout the section, $\eps\in(0,1)$ and $\omega\in\Omega$ are fixed.

First of all, conditions \eqref{est1}--\eqref{est_gamma} and \eqref{comp_beta} ensure that there exist
\begin{gather}
  \label{X_eps}
  X^\eps(\omega)\in L^\infty\left(0,T; H\right)\,,\\
  \label{Y_eps}
  Y^\eps(\omega)\in L^p(Q)^d\,,\\
  \label{eta_eps}
  \eta^\eps(\omega)\in L^q(Q)^d\,,\\ 
  \label{xi_eps}
  \xi^\eps(\omega)\in L^1\left(Q\right)
\end{gather}
and a sequence
$\{\lambda_n\}_{n\in\En}$ (which clearly depends on $\eps$ and $\omega$ as well) such that as $n\rarr\infty$
\begin{gather}
  \label{conv1''}
  X^\eps_{\lambda_n}(\omega)\weakstar X^\eps(\omega) \quad\text{in } L^\infty\left(0,T; H\right)\,,\\
  \label{conv2''}
  J_{\lambda_n}\left(\nabla X^\eps_{\lambda_n}(\omega)\right)\rarrw Y^\eps(\omega) \quad\text{in } L^p(Q)^d\,,\\
  \label{conv3''}
  \gamma_{\lambda_n}(\nabla X^\eps_\lambda(\omega))\rarrw \eta^\eps(\omega) \quad\text{in } L^q(Q)^d\,,\\
  \label{conv4''}
  \beta_{\lambda_n}(X^\eps_{\lambda_n}(\omega))\rarrw\xi^\eps(\omega) \quad\text{in } L^1\left(Q\right)
\end{gather}
and also as $\lambda\searrow0$ that
\begin{gather}
  \label{conv5}
  \lambda\gamma_{\lambda}(\nabla X^\eps_{\lambda}(\omega))\rarr0 \quad\text{in } L^2(Q)^d\,,\\
  \label{conv6}
  \lambda\nabla X^\eps_{\lambda}(\omega)\rarr0 \quad\text{in } L^2(Q)^d\,.
\end{gather}
In particular, since $\lambda^2|\gamma_\lambda(\nabla X^\eps_\lambda)|^2
=|\nabla X^\eps_\lambda-J_\lambda(\nabla X^\eps_\lambda)|^2$,
from \eqref{conv5} we have that
\[
  \int_{Q}{\left|\nabla X^\eps_{\lambda}-J_{\lambda}(\nabla X^\eps_{\lambda})\right|^2}(\omega)\rarr0 
  \quad\text{as }\lambda\searrow0\,,
\]
which together with \eqref{conv2''} implies that $\nabla X^\eps_{\lambda_n}(\omega)\rarrw Y^\eps$ in $L^2(Q)^d$; hence,
we deduce
\beq
  \label{X_eps_bis}
  X^\eps(\omega)\in L^p\left(0,T;V\right)\,,
\eeq
$Y^\eps=\nabla X^\eps$ and as a consequence (possibly renominating $\{\lambda_n\}_{n\in\En}$)
\begin{gather}
  \label{conv7''}
  J_{\lambda_n}\left(\nabla X^\eps_{\lambda_n}(\omega)\right)\rarrw \nabla X^\eps(\omega) \quad\text{in } L^p(Q)^d\,,\\
  \label{conv8''}
  \nabla X^\eps_{\lambda_n}(\omega)\rarrw\nabla X^\eps(\omega) \quad\text{in } L^2(Q)^d\,.
\end{gather}

The second step is to prove a strong convergence for $X^\eps_\lambda$.
To this purpose, equation \eqref{app1} can be rewritten on the path starting from $\omega$ as
\[
  \partial_t\left(X^\eps_\lambda-W^\eps_B\right)(t)-\Div\gamma_\lambda(\nabla X^\eps_\lambda(t))
  -\lambda\Delta X^\eps_\lambda(t)
  +\beta_\lambda(X^\eps_\lambda(t))=0 \quad\text{in } H^{-1}(D)
\]
for every $t\in[0,T]$: we estimate the different terms of the previous relation in the larger space $L^1(0,T;V_0^*)$.
Recalling that $L^1(D), H^{-1}(D), V^*\embed V_0^*$, using
the fact that $\l|-\Div v\r|_{V^*}\leq\l|v\r|_{L^q(D)}$ for every $v\in L^q(D)^d$ (thanks to definition \eqref{div})
and that $\l|-\Delta v\r|_{H^{-1}(D)}\leq\l|\nabla v\r|_{L^2(D)}$ for every $v\in\gH1$, 
using
conditions \eqref{est_grad}--\eqref{est_gamma} and \eqref{comp_beta}, we deduce that for every $\lambda\in(0,1)$
\begin{align*}
  \l|-\Div\gamma_\lambda(\nabla X^\eps_\lambda(\omega))\r|_{L^1(0,T;V_0^*)}\lesssim
  \l|\gamma_\lambda(\nabla X^\eps_\lambda(\omega))\r|_{L^q(Q)}
  \leq M_{\omega,\eps}\,,&\\
  \l|-\lambda\Delta X^\eps_\lambda\r|_{L^1(0,T;V_0^*)}\lesssim\lambda\l|\nabla X_\lambda^\eps\r|_{L^2(Q)}\leq M_{\omega, \eps}\,,&\\
  \l|\beta_\lambda(X^\eps_\lambda(\omega))\r|_{L^1(0,T;V_0^*)}\lesssim \l|\beta_\lambda(X^\eps_\lambda(\omega))\r|_{L^1(Q)}\leq
  M_{\omega.\eps}\,,&
\end{align*}
renominating the constant $M_{\omega,\eps}$ at each passage.
Hence, we deduce by difference that
\beq
  \label{est4}
  \l|\partial_t\left(X^\eps_\lambda-W^\eps_B\right)(\omega)\r|_{L^1(0,T;V_0^*)}\leq M_{\omega, \eps} \quad\text{for every } 
  \lambda\in(0,1)\,.
\eeq
At this point, we can recover a strong convergence using some classical compactness results with $\omega\in\Omega$ being fixed.
The proposition that we are going to use is the following (the reader can refer to 
\cite[Cor.~4, p.~85]{simon}).

\begin{prop}
  \label{simon}
  Let $A_1\cembed A_2\embed A_3$ be three Banach spaces and let $F\subseteq L^r(0,T;A_1)$ be a 
  bounded set such that $\frac{\partial F}{\partial t}:=\{\partial_t f: f\in F\}$ is bounded in $L^1(0,T;A_3)$ for a 
  given $r\geq1$. Then $F$ 
  is relatively compact in $L^r(0,T;A_2)$.
\end{prop}

In our setting, we make the natural choices $A_1=\gH1$,
$A_2=H$, $A_3=V_0^*$, $r=~2$ and $F=\{(X^\eps_{\lambda_n}-W^\eps_B)(\omega)\}_{n\in\En}$:
since by \eqref{conv8''} the family $\{X^\eps_{\lambda_n}(\omega)\}_{n\in\En}$ is bounded in $L^2(0,T;\gH1)$,
thanks also to \eqref{est4} we can apply
Proposition \ref{simon} to recover that the set $F$
is relatively compact in $L^2(0,T;H)$.
Hence, there exists $X^\eps_B(\omega)\in L^2(0,T;H)$
such that
\[
  (X^\eps_{\lambda_n}-W^\eps_B)(\omega)\rarr X^\eps_B(\omega) \quad\text{in } L^2(0,T;H)
   \quad\text{as } n\rarr\infty\,,
\]
possibly updating the sequence $\{\lambda_n\}_{n\in\En}$.
Using condition \eqref{conv1''} and the fact that $W^\eps_B$ is fixed with respect to $\lambda$,
we infer that
\[
  (X^\eps_{\lambda_n}-W^\eps_B)(\omega)\weakstar (X^\eps-W^\eps_B)(\omega) \quad\text{in } L^\infty(0,T;H)
  \quad\text{as } n\rarr\infty\,,
\]
and for uniqueness of the weak limit we have $X^\eps_B(\omega)=(X^\eps-W^\eps_B)(\omega)$ a.e.~in $Q$.
As a consequence, we have that
\beq
  \label{conv_fort}
  X^\eps_{\lambda_n}(\omega)\rarr X^\eps(\omega) \quad\text{in } L^2(0,T;H) \quad\text{as } n\rarr\infty\,.
\eeq

We are now ready to pass to the limit as $\lambda\searrow0$ in \eqref{app1}: 
in particular, we are going to show that for every $\eps\in(0,1)$
we have
\begin{gather}
  \label{app_eps}
  \begin{split}
    X^\eps(t)-\int_0^t{\Div \eta^\eps(s)\,ds}
    +\int_0^t{\xi^\eps(s)\,ds}&=X_0+\int_0^t{B^\eps(s)\,dW_s} \quad\text{in } V_0^*\,,\\
    &\text{for every } t\in[0,T]\,, \quad\Pi\text{-almost surely}\,,
  \end{split}\\
  \label{incl_eta_eps}
  \eta^\eps\in\gamma(\nabla X^\eps) \quad\text{a.e.~in } Q\,, \quad\Pi\text{-almost surely}\,,\\
  \label{incl_xi_eps}
  \xi^\eps\in\beta(X) \quad\text{a.e.~in } Q\,, \quad\Pi\text{-almost surely}\,,\\
  \label{j_l1}
  j(X^\eps)+j^*(\xi^\eps)\in L^1(Q)\,, \quad\Pi\text{-almost surely}\,.
\end{gather}

Firstly, let $\eps\in(0,1)$ and $\omega\in\Omega$ be fixed as usual.
Let $w\in V_0$ and recall the fact that $V_0\embed L^\infty(D)\cap V$:
then, thanks to \eqref{conv1''}, \eqref{conv3''}, \eqref{conv6} and \eqref{conv4''}, 
for almost every $t\in(0,T)$ we have
\begin{gather*}
  \int_D{X^\eps_{\lambda_n}(t)w}\rarr\int_D{X^\eps(t)w}\,,\\
  \int_0^t\int_D{\gamma_{\lambda_n}(\nabla X^\eps_{\lambda_n}(s))\cdot\nabla w\,ds}
  \rarr \int_0^t\int_D{\eta^\eps(s)\cdot\nabla w\,ds}\,,\\
  \lambda_n\int_0^t\int_D{\nabla X^\eps_{\lambda_n}(s)\cdot\nabla w\,ds}\rarr0\,,\\
  \int_0^t\int_D{\beta_{\lambda_n}(X^\eps_{\lambda_n}(s))w\,ds}\rarr\int_0^t\int_D{\xi^\eps(s)w\,ds}\,,
\end{gather*}
as $n\rarr\infty$. Hence,
taking these remarks into account,
letting $n\rarr\infty$ in equation \eqref{app1} evaluated with $\lambda_n$, we obtain exactly 
\[
  \begin{split}
    X^\eps(t)-\int_0^t{\Div\eta^\eps(s)\,ds}&+\int_0^t{\xi^\eps(s)\,ds}=X_0+\int_0^t{B^\eps(s)\,dW_s} \qquad\text{in } V_0^*\\
    &\text{for almost every } t\in(0,T)\,, \quad\Pi\text{-almost surely}\,.
  \end{split}
\]
Since all the terms except the first are continuous with respect to time, we deduce a posteriori that
$X^\eps(\omega)\in C^0\left([0,T]; V_0^*\right)$ $\Pi$-almost surely,
which together with the fact that $X^\eps(\omega)\in L^\infty(0,T;H)$ and the classical result on weak continuity
contained in \cite[Thm.~2.1]{strauss} implies
\beq
  \label{cont_X_eps'}
  X^\eps\in C^0_w\left([0,T]; H\right) \quad\Pi\text{-almost surely}\,.
\eeq 
Hence, the last integral relation holds for every $t\in[0,T]$ and \eqref{app_eps} is proved.

Secondly, let us show \eqref{incl_xi_eps}.
By \eqref{conv_fort} we can assume that
$X^\eps_{\lambda_n}(\omega)\rarr X^\eps(\omega)$ almost everywhere in $Q$ as $k\rarr\infty$,
from which, since $R_{\lambda_n}$ is a contraction, 
we deduce also that 
$R_{\lambda_n}X^\eps_{\lambda_n}(\omega)\rarr X^\eps(\omega)$ almost everywhere in $Q$.
Moreover, by \eqref{prop_res2} and \eqref{conv4''}, we also know that
$\beta_{\lambda_n}(X^\eps_{\lambda_n}(\omega))\in
\beta(R_{\lambda_n}X^\eps_{\lambda_n}(\omega))$
and $\beta_{\lambda_n}(X_{\lambda_n}(\omega))\rarrw\xi^\eps(\omega)$ in $L^1(Q)$.
Consequently, since
$\{\beta_{\lambda_n}(X_{\lambda_n}(\omega))X_{\lambda_n}(\omega)\}_{n\in\En}$ is
bounded in $L^1(Q)$ thanks to \eqref{est_j_star}, we can apply the result contained in \cite[Thm.~18, p.~126]{brezis3}
to infer \eqref{incl_xi_eps}.

Furthermore, by definition of $\beta_{\lambda_n}$ we have
$X^\eps-R_{\lambda_n}X^\eps_{\lambda_n}=(X^\eps-X^\eps_{\lambda_n})+\lambda_n\beta_{\lambda_n}(X^\eps_{\lambda_n})$,
so that thanks to \eqref{conv4''} and \eqref{conv_fort} we deduce
that 
$R_{\lambda_n}X^\eps_{\lambda_n}(\omega)\rarr X^\eps(\omega)$ in $L^1(Q)$:
hence, by the weak lower semicontinuity of the convex integrals and conditions \eqref{conv4''},
\eqref{prop_res2}, \eqref{convex2} and \eqref{est_j_star}, we have that
\[
  \begin{split}
  \int_Q&{\left[j(X^\eps(\omega))+j^*(\xi^\eps(\omega))\right]}\leq\liminf_{n\rarr\infty}
  \int_{Q}\left[j(R_{\lambda_n}X^\eps_{\lambda_n}(\omega))+j^*(\beta_{\lambda_n}(X_{\lambda_n})(\omega))\right]\\
  &=\liminf_{n\rarr\infty}\int_{Q}R_{\lambda_n}X^\eps_{\lambda_n}(\omega)\beta_{\lambda_n}(X_{\lambda_n}(\omega))\leq
  \liminf_{n\rarr\infty}\int_{Q}X^\eps_{\lambda_n}(\omega)\beta_{\lambda_n}(X_{\lambda_n}(\omega))
  \leq M_{\omega, \eps}\,,
  \end{split}
\]
so that also \eqref{j_l1} is proved. Let us also point out that 
conditions \eqref{incl_xi_eps} and \eqref{convex2} imply $\xi^\eps X^\eps=j(X^\eps)+j^*(\xi^\eps)$ almost everywhere
on $Q$, so that from the very last calculations, using the fact that $R_{\lambda_n}$ is a contraction
and the monotonicity of $\beta_\lambda$, we have
\beq
  \label{liminf_xi_eps}
  \xi^\eps(\omega) X^\eps(\omega)\in L^1(Q)\,, \quad
  \int_Q{\xi^\eps(\omega) X^\eps(\omega)}\leq\liminf_{n\rarr\infty}
  \int_Q{\beta_{\lambda_n}(X^\eps_{\lambda_n}(\omega))X^\eps_{\lambda_n}(\omega)}\,.
\eeq

Finally, let us show that \eqref{incl_eta_eps} holds: in the next passages,
we will omit to write $\omega$ to simplify notations.
From equation \eqref{test_lam_eps} evaluated at time $T$,
recalling conditions \eqref{conv1''}, \eqref{conv3''}, \eqref{conv4''}, \eqref{conv6},
\eqref{liminf_xi_eps} and \eqref{stoc_reg}, we get that
\[
   \begin{split}
   &\limsup_{n\rarr\infty}\int_Q{\gamma_{\lambda_n}(\nabla X^\eps_{\lambda_n})\cdot\nabla X^\eps_{\lambda_n}}
  =\frac{1}{2}\l|X_0\r|^2_H+\lim_{n\rarr\infty}\int_Q{\gamma_{\lambda_n}(\nabla X^\eps_{\lambda_n})\cdot \nabla W^\eps_B}\\
  &\qquad\quad+\lim_{n\rarr\infty}\lambda_n\int_Q{\nabla X^\eps_{\lambda_n}\cdot\nabla W^\eps_B}
  +\lim_{n\rarr\infty}\int_Q{\beta_{\lambda_n}(X^\eps_{\lambda_n})W^\eps_B}\\
  &\qquad\quad-\frac{1}{2}\liminf_{n\rarr\infty}\l|X_{\lambda_n}^\eps(T)-W^\eps_B(T)\r|^2_H
  -\lim_{n\rarr\infty}\lambda_n\l|\nabla X^\eps_{\lambda_n}\r|^2_H
  -\liminf_{n\rarr\infty}\int_Q{\beta_{\lambda_n}(X^\eps_{\lambda_n})X^\eps_{\lambda_n}}\\
  &\leq\frac{1}{2}\l|X_0\r|^2_H+\int_Q{\eta^\eps\cdot\nabla W^\eps_B}
  +\int_Q{\xi^\eps W^\eps_B}-\frac{1}{2}\l|X^\eps(T)-W^\eps_B(T)\r|^2_H
  -\int_Q{\xi^\eps X^\eps}\,.
  \end{split}
\]
At this point, thanks to conditions \eqref{app_eps}--\eqref{j_l1}, we can prove
that the following testing formula holds:
\beq
  \label{test_eps}
  \frac{1}{2}\l|X^\eps(T)-W^\eps_B(T)\r|^2_H+
  \int_Q{\eta^\eps\cdot\nabla(X^\eps-W^\eps_B)}
  +\int_Q{\xi^\eps(X^\eps-W^\eps_B)}=\frac{1}{2}\l|X_0\r|^2_H\,.
\eeq 
\begin{rmk}
The proof of \eqref{test_eps} relies on sharp approximations of
elliptic type and is very technical: hence, we omit it here in order not to make
the treatment heavier. The reader can refer to Appendix \ref{A_app} for a complete and rigorous proof of \eqref{test_eps}.
\end{rmk}
Hence, thanks to \eqref{test_eps}, the last set of inequalities can be read as
\[
  \limsup_{n\rarr\infty}\int_Q{\gamma_{\lambda_n}(\nabla X^\eps_{\lambda_n})\cdot\nabla X^\eps_{\lambda_n}}
    \leq\int_Q{\eta^\eps\cdot\nabla X^\eps}\,,
\]
from which, using the definition of $\gamma_{\lambda_n}$ and condition \eqref{conv5} we deduce that
\[
  \begin{split}
  \limsup_{n\rarr\infty}&\int_Q{\gamma_{\lambda_n}(\nabla X^\eps_{\lambda_n})\cdot J_{\lambda_n}\left(\nabla X^\eps_{\lambda_n}\right)}
  =\limsup_{n\rarr\infty}\int_Q{\left[\gamma_{\lambda_n}(\nabla X^\eps_{\lambda_n})\cdot\nabla X^\eps_{\lambda_n}-
  \lambda_n|\gamma_{\lambda_n}(\nabla X^\eps_{\lambda_n})|^2\right]}\\
  &=\limsup_{n\rarr\infty}\int_Q\gamma_{\lambda_n}(\nabla X^\eps_{\lambda_n})\cdot\nabla X^\eps_{\lambda_n}-
  \lim_{n\rarr\infty}\lambda_n\int_Q{\left|\gamma_{\lambda_n}(\nabla X^\eps_{\lambda_n})\right|^2}\leq
  \int_Q{\eta^\eps\cdot\nabla X^\eps}\,.
  \end{split}
\]
This last inequality together with \eqref{conv2''} and \eqref{conv3''} implies condition \eqref{incl_eta_eps}
thanks to the usual tools of monotone analysis.

\subsection{Measurability properties of the solutions}
\label{omega_reg}
In this section, we show that the solution components $X^\eps$, $\eta^\eps$ and $\xi^\eps$
constructed in the previous section have also some regularity with respect to $\omega$.
Moreover,
we prove uniform estimates with respect to $\eps$:
to this purpose, we will use the results of
Sections \ref{second} and \ref{third}, as well as natural lower semicontinuity properties.

First of all, note that, a priori, $X^\eps$, $\eta^\eps$ and $\xi^\eps$ are not even
measurable processes, because of the way they have been build (the sequence $\lambda_n$
could depend on $\omega$ as well). To show measurability,
we need to prove uniqueness for problem \eqref{app_eps}--\eqref{j_l1}.
Hence, let  $(X^\eps_1, \eta^\eps_1, \xi^\eps_1)$ and $(X^\eps_2, \eta^\eps_2, \xi^\eps_2)$ satisfy 
conditions \eqref{app_eps}--\eqref{j_l1}: taking the difference of \eqref{app_eps} and setting
$Y^\eps:=X^\eps_1-X^\eps_2$, $\zeta^\eps:=\eta^\eps_1-\eta^\eps_2$ and $\psi^\eps:=\xi^\eps_1-\xi^\eps_2$
we have
\[
  \begin{split}
    Y^\eps(t)-\int_0^t{\Div \zeta^\eps(s)\,ds}
    +\int_0^t{\psi^\eps(s)\,ds}=0
    \quad\text{for every } t\in[0,T]\,, \quad\Pi\text{-almost surely}\,.
  \end{split}
\]
Now, by convexity we have
$j(Y^\eps/2)+j^*(\psi^\eps/2)\leq \frac{1}{2}\left(j(X^\eps_1)+j(X^\eps_2)+j^*(\xi^\eps_1)+j^*(\xi^\eps_2)\right)$,
where the right-hand side is in $L^1(Q)$:
hence, using the same argument as in Appendix \ref{A_app} with $X_0=0$ and $B=0$, we infer that
\[
  \frac12\l|Y^\eps(t)\r|^2_H+\int_0^t\int_D\zeta^\eps(s)\cdot\nabla Y^\eps(s)\,ds+
  \int_0^t\int_D\psi^\eps(s)Y^\eps(s)\,ds=0\,.
\]
The monotonicity of $\gamma$ and $\beta$ implies that $Y^\eps=0$. Moreover, in view of \eqref{ip3},
\luca{$\gamma$ is a continuous function.} This implies that $\zeta^\eps=0$ and the first integral
expression becomes $\int_0^t\psi^\eps(s)\,ds=0$ for every $t\in[0,T]$, so that also $\psi^\eps=0$
and uniqueness is proved. 

At this point, we are ready to prove that the sequence $\{\lambda_n\}_{n\in\En}$ constructed 
in the previous section can be chosen independent of $\omega$: more precisely, we can prove that
for any sequence $\{\lambda_n\}_{n\in\En}$ decreasing to $0$, conditions \eqref{conv1''}--\eqref{conv4''}
and \eqref{conv7''}--\eqref{conv8''} hold.
Indeed, let $\{\lambda_n\}_{n\in\En}$ be any sequence decreasing to $0$ and fix $\omega\in\Omega$:
then, for every subsequence of $\{\lambda_n\}_{n\in\En}$ (which we still denote with the same symbol
\luca{for sake of simplicity}),
the estimates \eqref{est1}--\eqref{est_gamma} hold. Proceeding as in Section \ref{first_lim} and invoking the uniqueness,
we can then
extract a further sub-subsequence (depending on $\omega$) along which the same weak convergences
to $X^\eps$, $\eta^\eps$ and $\xi^\eps$ hold. This implies that the convergences \eqref{conv1''}--\eqref{conv4''}
and \eqref{conv7''}--\eqref{conv8''} are true
for the original sequence $\{\lambda_n\}_{n\in\En}$, which does not depend on $\omega$.

\luca{Now, let us prove some measurability properties of the processes $X^\eps$, $\eta^\eps$ and $\xi^\eps$.} 
\luca{First of all, since $X^\eps_{\lambda_n}\rarr X^\eps$ in $L^2(0,T;H)$ $\Pi$-almost surely, it is clear 
that $X^\eps$ is predictable (since so are $X^\eps_{\lambda_n}$ for every $n\in\En$).}
\luca{Secondly, let us focus on $\xi^\eps$: we prove that $\beta_{\lambda_n}(X^\eps_{\lambda_n})\rarrw\xi^\eps$
in $L^1(\Omega\times(0,T)\times D)$. To this aim, for any $g\in L^\infty(Q)$, setting
\[ 
  F_{\lambda_n}^\eps:=\int_Q\beta_{\lambda_n}(X^\eps_{\lambda_n})g\,, \qquad
  F^\eps:=\int_Q\xi^\eps g\,,
\]
we know that $F^\eps_{\lambda_n}\rarr F^\eps$ $\Pi$-almost surely:
let us show that $F^\eps_{\lambda_n}\rarrw F^\eps$ in $L^1(\Omega)$. Indeed,
for any $h\in L^\infty(\Omega)$, if we define
\[
  j^*_0(\cdot):=j^*\left(\cdot/M\right)\,, \quad M:=\frac1{(1\vee\l|g\r|_{L^\infty(Q)})(1\vee\l|h\r|_{L^\infty(\Omega)})}\,,
\]
by the Jensen inequality we have that
\[
  \begin{split}
  \exval\left[j^*_0(F_{\lambda_n}^\eps h)\right]&=
  \exval\left[j^*_0\left(\int_Q\beta_{\lambda_n}(X_{\lambda_n}^\eps)gh\right)\right]\\
  &\lesssim_{T,|D|}\exval\int_Qj^*_0(\beta_{\lambda_n}(X_{\lambda_n}^\eps)gh)\leq
  \int_{\Omega\times Q}j^*(|\beta_{\lambda_n}(X_{\lambda_n}^\eps)|)\,,
  \end{split}
\]
where the last term is bounded uniformly in $n$ by \eqref{est_j_star'}.
Consequently, since $j^*_0$ is still superlinear at infinity, by the de la Vallée-Poussin criterion,
we deduce that $\{F^\eps_{\lambda_n}h\}_{n\in\En}$ is uniformly integrable on $\Omega$:
taking also into account that $F^\eps_{\lambda_n}h\rarr F^\eps h$ $\Pi$-almost surely,
Vitali's convergence theorem ensures that $F_{\lambda_n}^\eps h\rarr F^\eps h$ in $L^1(\Omega)$.
Since this is true for any $h$ and $g$, this implies that $\beta_{\lambda_n}(X^\eps_{\lambda_n})\rarrw\xi^\eps$
in $L^1(\Omega\times(0,T)\times D)$.
By Mazur's Lemma
there is a sequence made up of convex combinations of
$\beta_{\lambda_n}(X_{\lambda_n}^\eps)$
which converge strongly $\xi^\eps$ in $L^1(Q)$, $\Pi$-almost surely.
This ensures that $\xi^\eps$ is predictable 
(since so are $\beta_\lambda(X^\eps_\lambda)$ for every $n$).
Finally, using a similar argument, one can show also that $\eta^\eps$ is adapted.}

It is now time to prove some uniform estimates with respect to $\eps$.
By \eqref{conv1''}--\eqref{conv4''}, \eqref{conv7''} and the estimates \eqref{est1'}--\eqref{est_gamma'}, using
the lower semicontinuity of the norm, we have
\begin{gather*}
  \l|X^\eps(\omega)\r|_{L^\infty(0,T; H)}
  \leq\liminf_{n\rarr\infty}\l|X^\eps_{\lambda_n}(\omega)\r|_{L^\infty(0,T;H)}\,,\\
  \l|\nabla X^\eps(\omega)\r|_{L^p(Q)}\leq
  \liminf_{n\rarr\infty}\l|J_{\lambda_n}\left(\nabla X^\eps_{\lambda_n}(\omega)\right)\r|_{L^p(Q)}\,,\\
  \l|\eta^\eps(\omega)\r|_{L^q(Q)}\leq
  \liminf_{n\rarr\infty}\l|\gamma_{\lambda_n}\left(\nabla X^\eps_{\lambda_n}(\omega)\right)\r|_{L^q(Q)}\,,\\
  \l|\xi^\eps(\omega)\r|_{L^1(Q)}\leq
  \liminf_{n\rarr\infty}\l|\beta_{\lambda_n}(X^\eps_{\lambda_n}(\omega))\r|_{L^1(Q)}\,.
\end{gather*}
Taking expectations and using \eqref{est1'}--\eqref{est_gamma'} and \eqref{comp_beta'},
the Fatou's lemma implies
\begin{gather*}
  \exval\l|X^\eps\r|^2_{L^\infty(0,T; H)}\leq\exval\left[\left(\liminf_{n\rarr\infty}\l|X^\eps_{\lambda_n}\r|_{L^\infty(0,T; H)}\right)^2\right]
  \leq\liminf_{n\rarr\infty}\l|X^\eps_{\lambda_n}\r|^2_{L^2(\Omega; L^\infty(0,T; H))} \leq N\,,\\
  \exval\l|\nabla X^\eps\r|_{L^p(Q)}^p\leq
  \exval\left[\left(\liminf_{n\rarr\infty}\l|J_{\lambda_n}\left(\nabla X^\eps_{\lambda_n}\right)\r|_{L^p(Q)}\right)^p\right]\leq
  \liminf_{n\rarr\infty}\l|J_{\lambda_n}\left(\nabla X^\eps_{\lambda_n}\right)\r|_{L^p(\Omega\times Q)}^p\leq N\,,\\
  \exval\l|\eta^\eps\r|_{L^q(Q)}^q\leq
  \exval\left[\left(\liminf_{n\rarr\infty}\l|\gamma_{\lambda_n}\left(\nabla X^\eps_{\lambda_n}\right)\r|_{L^q(Q)}\right)^q\right]\leq
  \liminf_{n\rarr\infty}\l|\gamma_{\lambda_n}\left(\nabla X^\eps_{\lambda_n}\right)\r|_{L^q(\Omega\times Q)}^q\leq N\,,\\
  \exval\l|\xi^\eps\r|_{L^1(Q)}\leq
  \exval\left[\liminf_{n\rarr\infty}\l|\beta_{\lambda_n}(X^\eps_{\lambda_n})\r|_{L^1(Q)}\right]\leq
  \liminf_{n\rarr\infty}\l|\beta_{\lambda_n}(X^\eps_{\lambda_n})\r|_{L^1(\Omega\times Q)}\leq N\,,
\end{gather*}
for a certain positive constant $N$ independent of $\eps$.
Hence, we have also proved that
\begin{gather}
  \label{X_eps_reg}
  X^\eps\in L^2\left(\Omega; L^\infty(0,T; H)\right)\cap L^p\left(\Omega\times(0,T); V\right)\,,\\
  \label{eta_xi_eps_reg}
  \eta^\eps\in L^q\left(\Omega\times(0,T)\times D\right)^d\,, \quad \xi^\eps\in L^1\left(\Omega\times(0,T)\times D\right)
\end{gather}
and that the following estimates hold:
\begin{gather}
  \label{est_eps1}
  \l|X^\eps\r|_{L^2(\Omega; L^\infty(0,T; H))\cap L^p(\Omega\times(0,T); V)}\leq N \quad\text{for every } \eps\in(0,1)\,,\\
  \label{est_eps2}
  \l|\eta^\eps\r|_{L^q(\Omega\times(0,T)\times D)}\leq N\quad\text{for every } \eps\in(0,1)\,,\\
  \label{est_eps3}
  \l|\xi^\eps\r|_{L^1(\Omega\times(0,T)\times D)}\leq N\quad\text{for every } \eps\in(0,1)\,.
\end{gather}

Moreover, since $\beta_{\lambda_n}(X^\eps_{\lambda_n})\rarrw\xi^\eps$ in $L^1(Q)$ as $n\rarr\infty$, $\Pi$-almost surely, 
by the weak lower semicontinuity of
the convex integral we have
\[
  \int_Q{j^*(\xi^\eps)}\leq\liminf_{n\rarr\infty}
  \int_Q{j^*\left(\beta_{\lambda_n}(X^\eps_{\lambda_n})\right)} \quad\Pi\text{-almost surely}\,:
\]
hence, thanks to the Fatou lemma and condition \eqref{est_j_star'}, we deduce that
\[
  \int_{\Omega\times Q}{j^*(\xi^\eps)}\leq
  \liminf_{n\rarr\infty}\int_{\Omega\times Q}{j^*\left(\beta_{\lambda_n}(X^\eps_{\lambda_n})\right)}\leq N\,,
\]
where $N$ is independent of $\eps$. Consequently, since $j^*$ is even thanks to \eqref{j_even},
we have that $\{j^*(\xi^\eps)\}_{\eps\in(0,1)}$ is bounded in $L^1(\Omega\times Q)$: hence,
since $j^*$ is superlinear at $\infty$, the classical results by De là Vallée-Poussin and the Dunford.Pettis theorem ensure that
\beq
  \label{comp_eps}
  \{\xi^\eps\}_{\eps\in(0,1)} \quad\text{is weakly relatively compact in } L^1(\Omega\times(0,T)\times D)\,.
\eeq
Similarly, $R_{\lambda_n}X^\eps_{\lambda_n}\rarr X^\eps$ in $L^1(Q)$
and $j(R_\lambda X^\eps_\lambda)\leq j(R_\lambda X^\eps_\lambda)
+j^*(\beta_\lambda(X^\eps_\lambda))=\beta_\lambda(X^\eps_\lambda)X^\eps_\lambda$: hence,
the weak lower semicontinuity of the convex integrals, Fatou's lemma and condition \eqref{est_j_star'} imply
\[
\int_{\Omega\times Q}{j(X^\eps)}\leq
\liminf_{n\rarr\infty}\int_{\Omega\times Q}{j\left(R_{\lambda_n}X^\eps_{\lambda_n}\right)}\leq
\sup_{\eps, \lambda\in(0,1)}\l|\beta_\lambda(X^\eps_\lambda)X^\eps_\lambda\r|_{L^1(\Omega\times Q)}\leq N\,.
\]
Taking these remarks into account, we have also obtained that
\beq
  \label{j_l1'}
  \l|j(X^\eps)\r|_{L^1(\Omega\times(0,T)\times D)}+\l|j^*(\xi^\eps)\r|_{L^1(\Omega\times(0,T)\times D)}\leq N \quad\text{for every }
  \eps\in(0,1)\,.
\eeq

\subsection{Passage to the limit as $\eps\searrow0$}
\label{second_lim}

In this section, we pass to the limit as $\eps\searrow0$ in the sub-prolem \eqref{app_eps}--\eqref{j_l1} and we recover 
global solutions to the original problem: to this end, the passage to the limit takes place also in probability, as we have
already anticipated.

First of all, thanks to \eqref{est_eps1}--\eqref{est_eps3}, we deduce that there exist
\begin{gather}
  \label{X_lim}
  X\in L^\infty\left(0,T;L^2(\Omega; H)\right)\cap L^p\left(\Omega\times(0,T);V\right)\,,\\
  \label{xi_eta_lim}
  \eta\in L^q\left(\Omega\times(0,T)\times D\right)^d\,, \qquad
  \xi\in L^1\left(\Omega\times(0,T)\times D\right)\,,
\end{gather}
and a sequence $\{\eps_n\}_{n\in\En}$ with $\eps_n\searrow0$ as $n\rarr\infty$ such that
\begin{gather}
  \label{conv1'}
  X^{\eps_n}\weakstar X \quad\text{in }L^\infty\left(0,T;L^2(\Omega; H)\right)\,,\\
  \label{conv2'}
  X^{\eps_n}\rarrw X \quad\text{in }L^p\left(\Omega\times(0,T); V)\right)\,,\\
  \label{conv3'}
  \eta^{\eps_n}\rarrw\eta \quad\text{in } L^q\left(\Omega\times(0,T)\times D\right)^d\,,\\
  \label{conv4'}
  \xi^{\eps_n}\rarrw\xi \quad\text{in } L^1\left(\Omega\times(0,T)\times D\right)\,.
\end{gather}

Let us prove a strong convergence for $X^\eps$: given $\eps,\delta\in(0,1)$,
consider equation \eqref{app_eps} evaluated for $\eps$ and $\delta$. Then, taking the difference we have
\[
  \begin{split}
    X^\eps(t)-X^\delta(t)&-
    \int_0^t{\Div(\eta^\eps(s)-\eta^\delta(s))\,ds}
    +\int_0^t{\left(\xi^\eps(s)-\xi^\delta(s)\right)\,ds}\\
    &=\int_0^t{(B^\eps(s)-B^\delta(s))\,dW_s}
    \quad\text{in } V_0^*
    \quad\text{for every } t\in[0,T]\,, \quad\Pi\text{-a.s}\,.
    \end{split}
\]
Now, notice that thanks to \eqref{j_even} and the convexity of $j$ and $j^*$, we have
\[
  j\left(\frac{X^\eps-X^\delta}{2}\right)+j^*\left(\frac{\xi^\eps-\xi^\delta}{2}\right)\leq
  \frac{1}{2}\left(j(X^\eps)+j(X^\delta)+j^*(\xi^\eps)+j^*(\xi^\delta)\right)\,,
\]
where the term on the right hand side is in $L^1(\Omega\times(0,T)\times D)$ thanks to \eqref{j_l1'}:
hence, recalling also condition \eqref{cont_X_eps'} we can apply Proposition \ref{prop_ito}
with the choices $Y=X^\eps-X^\delta$, $f=\eta^\eps-\eta^\delta$, $g=\xi^\eps-\xi^\delta$, $T=B^\eps-B^\delta$ and 
$\alpha=1/2$
to infer that
\[
  \begin{split}
  &\frac{1}{2}\l|X^\eps(t)-X^\delta(t)\r|^2_H
  +\int_0^t\int_{D}
  {\left(\eta^\eps(s)-\eta^\delta(s)\right)\cdot\left(\nabla X^\eps(s)-\nabla X^\delta(s)\right)\,ds}\\
  &\qquad\qquad\qquad+\int_0^t\int_{D}{\left(\xi^\eps(s)-\xi^\delta(s)\right)
  \left(X^\eps(s)-X^\delta(s)\right)\,ds}\\
  &=\frac{1}{2}\int_0^t{\l|B^\eps(s)-B^\delta(s)\r|_{\cL_2(U,H)}^2\,ds}+
  \int_0^t\left((X^\eps-X^\delta)(s), (B^\eps-B^\delta)(s)\,dW_s\right)
  \end{split}
\]
for every $t\in[0,T]$, $\Pi$-almost surely. Now, proceeding exactly as in Section \ref{second},
we take the supremum in $t$ and expecations,
use the monotonicity of $\gamma$ and $\beta$ together with \eqref{incl_eta_eps}--\eqref{incl_xi_eps}
and the Davis inequality, so that we have
\[
  \l|X^\eps-X^\delta\r|^2_{L^2(\Omega; L^\infty(0,T; H))}\lesssim\l|B^\eps-B^\delta\r|^2_{L^2(\Omega\times(0,T); \cL_2(U,H))}
\]
for every $\eps,\delta\in(0,1)$: taking into account \eqref{B_eps1}, this implies that the sequence
$\{X^\eps\}_{\eps\in(0,1)}$ is Cauchy in $L^2(\Omega; L^\infty(0,T;H))$, so that by \eqref{conv1'} we deduce
\beq
  \label{X_reg}
  X\in L^2\left(\Omega; L^\infty(0,T; H)\right)
\eeq
and
\beq
  \label{conv_fort_eps}
  X^\eps\rarr X \quad\text{in } 
  L^2\left(\Omega;L^\infty(0,T; H)\right)\,, \quad\text{as } \eps\searrow0\,.
\eeq

We are now ready to pass to the limit in equation \eqref{app_eps}: to this purpose,
fix $w\in V_0$
(recall that $V_0\embed L^\infty(D)\cap V$).
Then, thanks to \eqref{conv_fort_eps}, \eqref{conv2'}--\eqref{conv4'}
and \eqref{B_eps1},
for every $t\in[0,T]$ we have as $n\rarr\infty$ that
\begin{gather*}
  \exval\left[\operatorname{ess}\sup_{t\in(0,T)}\left|\int_D{X^{\eps_n}(t)w}-\int_D{X(t)w}\right|\right]\rarr0\,,\\
  \exval\left[\int_0^t\int_D\eta^{\eps_n}\cdot \nabla w\,ds\right]
  \rarr\exval\left[\int_0^t\int_D{\eta\cdot\nabla w\,ds}\right]\,,\\
  \exval\left[\int_0^t\int_D{\xi^{\eps_n}(s)w\,ds}\right]\rarr\exval\left[\int_0^t\int_D{\xi(s)w\,ds}\right]\,,\\
  \exval\left[\int_0^t{\left(w, B^{\eps_n}(s)\,dW_s\right)}\right]\rarr\exval\left[\int_0^t{\left(w, B(s)\,dW_s\right)}\right]\,,
\end{gather*}
so that evaluating \eqref{app_eps} with $\eps_n$
and letting $n\rarr\infty$, we deduce
\[
  \begin{split}
    X(t)-\int_0^t{\Div\eta(s)\,ds}+\int_0^t{\xi(s)\,ds}=&X_0+\int_0^t{B(s)\,dW_s} \qquad\text{in } V_0^*\,,\\
    &\text{for almost every } t\in(0,T)\,, \quad\Pi\text{-almost surely}\,.
  \end{split}
\]
Since all the terms except the first have $\Pi$-almost surely continuous paths in $V_0^*$, we have {\em a posteriori} that
$X\in C^0\left([0,T]; V_0^*\right)$ $\Pi$-almost surely,
which together with \eqref{X_reg} and the result contained in
\cite[Thm.~2.1]{strauss} implies
\beq
  \label{cont_X'}
  X\in C^0_w\left([0,T]; H\right) \quad\Pi\text{-almost surely}\,,
\eeq
so that the integral relation holds for every $t\in[0,T]$
and \eqref{sol1}--\eqref{var_form} are proved.
Furthermore, for every $t\in[0,T]$ and $\Pi$-almost surely,
all the terms in \eqref{var_form} except $\int_0^t\eta(s)\,ds$ are in $L^1(D)$
and all the terms except $\int_0^t\xi(s)\,ds$ are in $V^*$, so that by difference
the integral relation holds in $L^1(D)\cap V^*$.

At this point, let us focus on \eqref{incl_xi} and \eqref{j_integr}.
By \eqref{conv_fort_eps}, we may assume that $X^{\eps_n}\rarr X$ almost everywhere in $\Omega\times Q$;
moreover, by \eqref{incl_xi_eps}, \eqref{j_l1'} and \eqref{prop_res2} we have 
\[
\int_{\Omega\times Q}{\xi^\eps X^\eps}=\int_{\Omega\times Q}\left(j(X^\eps)+j^*(\xi^\eps)\right)\leq N\,,
\]
where $N>0$ is independent of $\eps$. Hence, $\{\xi^\eps X^\eps\}_{\eps\in(0,1)}$ is bounded in $L^1(\Omega\times Q)$,
and recalling also \eqref{conv4'} we can apply the result contained in \cite[Thm.~18, p.~126]{brezis3}
to infer that \eqref{incl_xi} holds. Moreover, thanks to
conditions \eqref{conv_fort_eps}, \eqref{conv4'} and \eqref{j_l1'}, using the weak lower
semicontinuity of the convex integrals we have that
\[
  \int_{\Omega\times Q}{\left(j(X)+j^*(\xi)\right)}\leq
  \liminf_{n\rarr\infty}\int_{\Omega\times Q}{\left(j(X^{\eps_n})+j^*(\xi^{\eps_n})\right)}\leq N\,,
\]
so that \eqref{j_integr} is proved. Let us also point out that from the last inequality,
thanks to \eqref{incl_xi}, \eqref{incl_xi_eps} and \eqref{convex2} we obtain
\beq
  \label{lim_inf}
  \int_{\Omega\times Q}{\xi X}\leq\liminf_{n\rarr\infty}\int_{\Omega\times Q}{\xi^{\eps_n}X^{\eps_n}}\,.
\eeq

The next thing that we need to prove is condition \eqref{incl_eta}. To this end, thanks to
\eqref{X_eps}--\eqref{xi_eps}, \eqref{X_eps_bis},
 \eqref{app_eps}--\eqref{j_l1} and \eqref{cont_X_eps'},
we can apply Proposition \ref{prop_ito} to infer that for every $t\in[0,T]$
\[
  \begin{split}
  \frac{1}{2}\l|X^{\eps_n}(t)\r|^2_{L^2(\Omega; H)}&
  +\int_0^t\int_{\Omega\times D}
  {\eta^{\eps_n}(s)\cdot\nabla X^{\eps_n}(s)\,ds}
  +\int_0^t\int_{\Omega\times D}{\xi^{\eps_n}(s)X^{\eps_n}(s)\,ds}\\
  &=\frac{1}{2}\l|X_0\r|_{L^2(\Omega; H)}^2+\frac{1}{2}\int_0^t{\l|B^{\eps_n}(s)\r|_{L^2(\Omega; \cL_2(U,H))}^2\,ds}\,,
  \end{split}
\]
from which, thanks to \eqref{conv_fort_eps}, \eqref{lim_inf} and \eqref{B_eps2}, we have $\Pi$-almost surely that
\[
\begin{split}
  &\limsup_{n\rarr\infty}\int_{\Omega\times Q}{\eta^{\eps_n}\cdot\nabla X^{\eps_n}}=
  \frac12\l|X_0\r|_{L^2(\Omega; H)}^2+\frac12\lim_{n\rarr\infty}\l|B^{\eps_n}\r|_{L^2(\Omega\times(0,T); \cL_2(U,H))}^2\\
  &\qquad\qquad\qquad\qquad
  -\frac{1}{2}\liminf_{n\rarr\infty}\l|X^{\eps_n}(T)\r|^2_{L^2(\Omega;H)}
  -\liminf_{n\rarr\infty}\int_{\Omega\times Q}{\xi^{\eps_n}X^{\eps_n}}\\
  &\leq\frac{1}{2}\l|X_0\r|_{L^2(\Omega;H)}^2+\frac12\l|B\r|_{L^2(\Omega\times(0,T); \cL_2(U,H))}^2
  -\frac{1}{2}\l|X(T)\r|^2_{L^2(\Omega;H)}
  -\int_{\Omega\times Q}{\xi X}\,.
  \end{split}
\]
Now, thanks to conditions \eqref{X_lim}--\eqref{xi_eta_lim}, \eqref{cont_X'}, \eqref{var_form} and \eqref{j_integr},
we can apply a second time
Proposition \ref{prop_ito} with the choices $Y=X$, $f=\eta$, $g=\xi$ and $T=B$: hence, the right hand side of the last 
set of inequality is exactly $\int_{\Omega\times Q}{\eta\cdot\nabla X}$, so that we have
\[
  \limsup_{n\rarr\infty}\int_{\Omega\times Q}{\eta^{\eps_n}\cdot\nabla X^{\eps_n}}\leq
  \int_{\Omega\times Q}{\eta\cdot\nabla X}\,.
\]
This condition together with \eqref{conv2'}--\eqref{conv3'} and \eqref{incl_eta_eps} implies exactly \eqref{incl_eta}.

Finally, let us show that $X$ and $\xi$ are predictable processes, and $\eta$ is adapted.
At the end of Section~\ref{omega_reg} we checked that $X^\eps$ and $\xi^\eps$ are predictable,
and $\eta^\eps$ is adapted, for every $\eps\in(0,1)$.
Now, from \eqref{conv_fort_eps} it immediately follows that also $X$ is predictable.
Moreover, by conditions \eqref{conv3'}--\eqref{conv4'} and Mazur's Lemma we can recover strong convergences 
for some suitable convex combinations of $\{\eta^{\eps_n}\}$ and $\{\xi^{\eps_n}\}$: since these are still
adapted and predicable, respectively, we can easily infer that $\eta$ is adapted and $\xi$ is predictable.
This completes the proof.

\subsection{The further existence result}
In this section we prove the last part of Theorem \ref{th1}, in which
condition \eqref{ip3} is not assumed anymore. \luca{The idea is to to pass 
to the limit in a different way, using only the estimates in expectations and
avoiding the pathwise arguments.}

\luca{For any $\lambda\in(0,1)$, consider the approximated problem
\begin{gather*}
  dX_\lambda-\Div\gamma_\lambda(\nabla X_\lambda)\,dt - \lambda\Delta X_\lambda\,dt+\beta_\lambda(X_\lambda)\,dt \ni B\,dW_t\,:
\end{gather*}
the classical variational approach in the Gelfand triple $H^1_0(D)\embed H\embed H^{-1}(D)$ ensures the existence of
the approximated solutions
\[
  X_\lambda\in L^2\left(\Omega; C^0([0,T]; H)\right)\cap L^2\left(\Omega\times(0,T); H^1_0(D)\right)\,.
\]}
\luca{Using It\^o's formula
and proceeding as in Sections \ref{second} and \ref{third}, it is not difficult to prove that there exist a positive constant
$N$, independent of $\lambda$, such that
\begin{gather*}
  \l|X_\lambda\r|_{L^2(\Omega; L^\infty(0,T; H))}\leq N\,,\qquad
  \l|J_\lambda(\nabla X_\lambda)\r|_{L^p(\Omega\times(0,T)\times D)}\leq N\,,\\
  \l|\gamma_\lambda(\nabla X_\lambda)\r|_{L^q(\Omega\times(0,T)\times D)}\leq N\,,\\
  \{\beta_\lambda(X_\lambda)\}_{\lambda\in(0,1)} \quad\text{is weakly relatively compact in } L^1(\Omega\times(0,T)\times D)\,,\\
  \l|j(X_\lambda)\r|_{L^1(\Omega\times(0,T)\times D)}+\l|j^*(\beta_\lambda(X_\lambda))\r|_{L^1(\Omega\times(0,T)\times D)}\leq N\,,\\
  \lambda^{1/2}\l|\nabla X_\lambda\r|_{L^2(\Omega\times(0,T); H)}\leq N\,,\\
  \lambda^{1/2}\l|\gamma_\lambda(\nabla X_\lambda)\r|_{L^2(\Omega\times(0,T)\times D)}\leq N\,.
\end{gather*}}
We deduce that there exist
\begin{gather*}
  X \in L^\infty\left(0,T; L^2(\Omega; H)\right)\cap L^p\left(\Omega\times(0,T); V\right)\,,\\
  \eta\in L^q\left(\Omega\times(0,T)\times D\right)^d\,,\qquad
  \xi\in L^1\left(\Omega\times(0,T)\times D\right)\,,
\end{gather*}
and a sequence $\{\lambda_n\}_{n\in\En}$ decreasing to $0$ such that, as $n\rarr\infty$,
\begin{gather*}
X_{\lambda_n}\weakstar X \quad\text{in } L^\infty\left(0,T; L^2(\Omega; H)\right)\,,\qquad
J_{\lambda_n}\left(\nabla X_{\lambda_n}\right)\rarrw \nabla X \quad\text{in }  L^p\left(\Omega\times(0,T)\times D\right)^d\,,\\
\gamma_{\lambda_n}(\nabla X_{\lambda_n})\rarrw\eta \quad\text{in } L^q\left(\Omega\times(0,T)\times D\right)^d\,,\qquad
\beta_{\lambda_n}(X_{\lambda_n})\rarrw\xi \quad\text{in } L^1\left(\Omega\times(0,T)\times D\right)\,.
\end{gather*}
Fix $w\in L^\infty(\Omega; V_0)$: then, since the four last convergences imply that
$X_{\lambda_n}(t)\rarrw X(t)$ in $L^2(\Omega; H)$ for almost every $t\in(0,T)$, we have, as $n\rarr\infty$,
\begin{gather*}
  \int_{\Omega\times D}{X_{\lambda_n}(t)w}\rarr\int_{\Omega\times D}X(t)w\,,\\
  \luca{\int_0^t\int_{\Omega\times D}\gamma_{\lambda_n}(\nabla X_{\lambda_n})\cdot\nabla w\rarr\int_0^t\int_{\Omega\times D}\eta\cdot\nabla w\,,}
  \qquad
  \int_0^t\int_{\Omega\times D}\beta_{\lambda_n}(X_{\lambda_n})w\rarr\int_0^t\int_{\Omega\times D}\xi w
\end{gather*}
for almost every $t\in(0,T)$. Hence, letting $n\rarr\infty$, we get, for almost every $t\in(0,T)$,
\[
  X(t)-\int_0^t\Div\eta(s)\,ds+\int_0^t\xi(s)\,ds=X_0+\int_0^tB(s)\,dW_s \quad\text{in }V_0^*\,,\quad\Pi\text{-almost surely}\,:
\]
since all the terms except the first are continuous with values in $L^1(\Omega; V_0^*)$, we infer also that
$X\in C^0([0,T]; L^1(\Omega; V_0^*))$ and the integral relation holds for every $t\in[0,T]$. Moreover, 
since we also have $X\in L^\infty(0,T; L^2(\Omega; H))$, 
by \cite[Thm.~2.1]{strauss} we can infer that $X\in C^0_w([0,T]; L^2(\Omega; H))$.

Secondly, using the weak lower semicontinuity of the convex integrals and the estimates on 
$j(X_\lambda)$ and $j^*(\beta_\lambda(X_\lambda))$, it is immediate to check that $j(X)+j^*(\xi)\in L^1(\Omega\times Q)$.
Furthermore, as we did at the end of Section \ref{second_lim}, using Mazur's lemma, we deduce also that
$X$ and $\xi$ are predictable, and $\eta$ is adapted.

The last thing that we have to check is that $\eta\in\gamma(\nabla X)$ and $\xi\in\beta(X)$ a.e.~in $\Omega\times Q$.
To this aim, by the second part of Proposition \ref{prop_ito},
\luca{using the notation $\eta_\lambda:=\gamma_\lambda(\nabla X_\lambda)$}, we have that, for every $t\in[0,T]$,
\[
  \begin{split}
    \frac{1}{2}&\l|X_\lambda(t)\r|_{L^2(\Omega; H)}^2+\int_0^t\int_{\Omega\times D}{\eta_\lambda(s)\cdot\nabla X_\lambda(s)\,ds}+
    \int_0^t\int_{\Omega\times D}{\beta_\lambda(X_\lambda)(s)X_\lambda(s)\,ds}\\
    &\qquad\qquad=\frac{1}{2}\l|X_0\r|_{L^2(\Omega; H)}^2+
    \frac{1}{2}\int_0^t{\l|B(s)\r|_{L^2(\Omega; \cL_2(U,H))}^2}\,ds
    \end{split}
\]
and
\[
    \begin{split}
    \frac{1}{2}&\l|X(t)\r|_{L^2(\Omega; H)}^2+\int_0^t\int_{\Omega\times D}{\eta(s)\cdot\nabla X(s)\,ds}+
    \int_0^t\int_{\Omega\times D}{\xi(s)X(s)\,ds}\\
    &\qquad\qquad=\frac{1}{2}\l|X_0\r|_{L^2(\Omega; H)}^2+
    \frac{1}{2}\int_0^t{\l|B(s)\r|_{L^2(\Omega; \cL_2(U,H))}^2}\,ds\,.
    \end{split}
\]
We deduce that
\[
  \begin{split}
  \limsup_{n\rarr\infty}&\left[\int_{\Omega\times Q}{\eta_{\lambda_n}\cdot\nabla X_{\lambda_n}}+
    \int_{\Omega\times Q}{\beta_{\lambda_n}(X_{\lambda_n})X_{\lambda_n}}\right]\\
    &=\frac{1}{2}\l|X_0\r|_{L^2(\Omega; H)}^2+\frac{1}{2}\int_0^T{\l|B(s)\r|_{L^2(\Omega; \cL_2(U,H))}^2}\,ds-
    \frac12\liminf_{n\rarr\infty}\l|X_{\lambda_n}(T)\r|^2_{L^2(\Omega; H)}\\
    &\leq\frac{1}{2}\l|X_0\r|_{L^2(\Omega; H)}^2+\frac{1}{2}\int_0^T{\l|B(s)\r|_{L^2(\Omega; \cL_2(U,H))}^2}\,ds-
    \frac12\l|X(T)\r|^2_{L^2(\Omega; H)}\\
    &=\int_{\Omega\times Q}{\eta\cdot\nabla X}+\int_{\Omega\times Q}{\xi X}\,.
  \end{split}
\]
Let us identify $\Ar^d\times\Ar$ with $\Ar^{d+1}$, indicate the generic element in $\Ar^{d+1}$ as a couple
$(x,y)$, where $x\in\Ar^d$ and $y\in\Ar$, and use the symbol $\bullet$
for the usual scalar product in $\Ar^{d+1}$.
Consider the proper, convex and lower semicontinuous function $\Phi: \Ar^{d+1}\rarr[0,+\infty)$
given by $\Phi(x,y):=k(x)+j(y)$, $(x,y)\in\Ar^{d+1}$: then the subdifferential of $\Phi$ is the operator 
$\Xi:\Ar^{d+1}\rarr2^{\Ar^{d+1}}$ given by $\Xi(x,y)=\{(u,v)\in\Ar^{d+1}: u\in\gamma(x), v\in\beta(y)\}$.
Hence, recalling that 
$\beta_\lambda(X_\lambda)R_\lambda X_\lambda=\beta_\lambda(X_\lambda)X_\lambda-\lambda|\beta_\lambda(X_\lambda)|^2
\leq\beta_\lambda(X_\lambda)X_\lambda$ \luca{and similarly
$\eta_\lambda\cdot J_\lambda(\nabla X_\lambda)=\eta_\lambda\cdot\nabla X_\lambda-\lambda|\eta_\lambda|^2$,}
we have proved that
\[
  \luca{\limsup_{n\rarr\infty}\int_{\Omega\times Q}(\eta_{\lambda_n}, \beta_{\lambda_n}(X_{\lambda_n}))
  \bullet(J_{\lambda_n}(\nabla X_{\lambda_n}), R_{\lambda_n}X_{\lambda_n})\leq
  \int_{\Omega\times Q}(\eta,\xi)\bullet(\nabla X, X)\,,}
\]
allowing us to infer that $(\eta,\xi)\in\Xi(\nabla X, X)$,
i.e.~that $\eta\in\gamma(\nabla X)$ and $\xi\in\beta(X)$ a.e.~in $\Omega\times Q$,
thanks to the classical results of convex analysis.


\section{Continuous dependence on the initial datum with additive noise}
\setcounter{equation}{0}
\label{cont_dep}

This section is devoted to the proof of the continuous dependence and uniqueness results contained in Theorem \ref{th2}.
The main tool that we use is the generalized It\^o formula contained in Proposition \ref{prop_ito}.

We start assuming \eqref{ip3}: let $(X_0^1, B_1)$, $(X_0^2, B_2)$, $(X_1, \eta_1, \xi_1)$, $(X_2, \eta_2, \xi_2)$ be as in 
Theorem \ref{th2}.
Then, writing relation \eqref{var_form} for $(X_1, \eta_1, \xi_1, X_0^1, B_1)$ and $(X_2, \eta_2,  \xi_2, X_0^2, B_2)$ 
and taking the difference, 
$\Pi$-almost surely we obtain
\[
  \begin{split}
    X_1(t)-X_2(t)&-\int_0^t{\Div\left[\gamma(\nabla X_1(s))-\gamma(\nabla X_2(s))\right]\,ds}+\int_0^t{(\xi_1(s)-\xi_2(s))\,ds}\\
    =&X_0^1-X_0^2+\int_0^t{(B_1(s)-B_2(s))\,dW_s} \qquad\text{for every } t\in[0,T]\,.
  \end{split}
\]
Now, we note that thanks to \eqref{j_integr} and \eqref{j_even}, for $i=1,2$ we have
\[
  j\left(\frac{X_1-X_2}{2}\right)+j^*\left(\frac{\xi_1-\xi_2}{2}\right)\leq\frac{1}{2}\left[j(X_1)+j(X_2)+j^*(\xi_1)+j(\xi_2)\right]\,,
\]
where the right hand side is in $L^1(\Omega\times(0,T)\times D)$: hence, we can apply Proposition \ref{prop_ito}
with the choices $Y=X_1-X_2$, $f=\eta_1-\eta_2$, $g=\xi_1-\xi_2$, $T=B_1-B_2$ and $\alpha=1/2$ in order to infer that
for every $t\in[0,T]$
\[
  \begin{split}
  &\frac{1}{2}\l|X_1(t)-X_2(t)\r|^2_{H}
  +\int_0^t\int_{D}
  {\left(\eta_1(s)-\eta_2(s)\right)\cdot\left(\nabla X_1(s)-\nabla X_2(s)\right)\,ds}\\
  &\qquad\qquad\qquad\qquad+\int_0^t\int_{D}{\left(\xi_1(s)-\xi_2(s)\right)
  \left(X_1(s)-X_2(s)\right)\,ds}\\
  &=\frac{\l|X_0^1-X_0^2\r|_{H}^2}{2}+
  \int_0^t{\frac{\l|(B_1-B_2)(s)\r|_{\cL_2(U,H)}^2}{2}\,ds}
  +\int_0^t\left((X_1-X_2)(s), (B_1-B_2)(s)\,dW_s\right)\,.
  \end{split}
\]
Hence, taking into account \eqref{incl_eta}--\eqref{incl_xi} and the monotonicity of $\gamma$ and $\beta$,
we obtain
\[
  \begin{split}
  \l|X_1(t)-X_2(t)\r|^2_{H}&\leq
  \l|X_0^1-X_0^2\r|_{H}^2+\int_0^t{\l|B_1(s)-B_2(s)\r|_{\cL_2(U,H)}^2\,ds}\\
  &+2\sup_{t\in[0,T]}\left|\int_0^t\left((X_1-X_2)(s), (B_1-B_2)(s)\,dW_s\right)\right|\,;
  \end{split}
\]
moreover,
proceeding exactly as in Section \ref{second}, taking the supremum in $t\in[0,T]$ in the last expression and then expectations,
thanks to the Davis inequality and the Young inequality, we easily obtain
\[
  \begin{split}
  \l|X_1-X_2\r|^2_{L^2(\Omega; L^\infty(0,T; H))}&\lesssim
  \l|X_0^1-X_0^2\r|_{L^2(\Omega; H)}^2+\l|B_1-B_2\r|_{L^2(\Omega\times(0,T); \cL_2(U,H))}^2\\
  &+\frac12\l|X_1-X_2\r|^2_{L^2(\Omega; L^\infty(0,T; H))}
  \end{split}
\]
from which \eqref{var_dep} follows. Finally, if $X_0^1=X_0^2$ and $B_1=B_2$, we 
immediately get $X_1=X_2$: substituting in the difference of the respective equations \eqref{var_form}
we have $\int_0^t\left(-\Div(\eta_1(s)-\eta_2(s))+(\xi_1(s)-\xi_2(s))\right)\,ds=0$ for every $t$.
Relying now on hypothesis \eqref{ip3} and proceeding as in Section \ref{omega_reg}, we easily
get also $\eta_1=\eta_2$ and $\xi_1=\xi_2$.

Let us prove now the second part of Theorem \ref{th2}, in which condition \eqref{ip3} is not assumed.
By the second part of Theorem \ref{th1}, we have that, for every $t\in[0,T]$,
\[
  \begin{split}
    X_1(t)-X_2(t)&-\int_0^t{\Div\left[\gamma(\nabla X_1(s))-\gamma(\nabla X_2(s))\right]\,ds}+\int_0^t{(\xi_1(s)-\xi_2(s))\,ds}\\
    =&X_0^1-X_0^2+\int_0^t{(B_1(s)-B_2(s))\,dW_s} \qquad\Pi\text{-almost surely}\,:
  \end{split}
\]
hence, using the second part of Proposition \ref{prop_ito}, we infer that, for every $t\in[0,T]$,
\[
    \begin{split}
  \frac{1}{2}\l|X_1(t)-X_2(t)\r|&^2_{L^2(\Omega; H)}
  +\int_0^t\int_{\Omega\times D}
  {\left(\eta_1(s)-\eta_2(s)\right)\cdot\left(\nabla X_1(s)-\nabla X_2(s)\right)\,ds}\\
  &\qquad\qquad+\int_0^t\int_{\Omega\times D}{\left(\xi_1(s)-\xi_2(s)\right)
  \left(X_1(s)-X_2(s)\right)\,ds}\\
  &=\frac12\l|X_0^1-X_0^2\r|_{L^2(\Omega; H)}^2+
  \frac12\int_0^t{\l|(B_1-B_2)(s)\r|_{L^2(\Omega; \cL_2(U,H))}^2\,ds}\,,
  \end{split}
\]
which together with the monotonicity of $\gamma$ and $\beta$ implies \eqref{weak_dep}.
Finally,  if $X_0^1=X_0^2$ and $B_1=B_2$, we 
have $X_1=X_2$ and $\int_0^t\left(-\Div(\eta_1(s)-\eta_2(s))+(\xi_1(s)-\xi_2(s))\right)\,ds=0$ for every $t$, as before,
so that $-\Div\eta_1+\xi_1=-\Div\eta_2+\xi_2$.

\section{Well-posedness with multiplicative noise}
\label{mult}

In this section, we prove the main theorem of the work, which ensures that the the original problem is well-posed
also with multiplicative noise. Let us describe the approach that we will follow.

The main idea is to prove existence of solutions proceeding step-by-step: we introduce a parameter $\tau>0$,
we prove using contraction estimates that we are able to recover some solutions on each subinterval 
$[0,\tau], [\tau, 2\tau], \ldots [n\tau, (n+1)\tau], \ldots$ provided that $\tau$ is chosen sufficiently small,
and finally we paste together each solution on the whole interval $[0,T]$.
In this sense, the main point of the argument is to prove that such a value of $\tau$ can be chosen uniformly with respect to $n$,
so that the procedure stops when we reach the final time $T$ (in a finite number of steps).

\subsection{Existence}

In this section we prove the two existence results contained in Theorem \ref{thm3}.
We start from the first one, i.e.~assuming \eqref{ip3}.
First of all,
for every $a,b\in[0,T]$ with $b>a$ and for any progressively measurable process $Y\in L^2(\Omega\times(0,T)\times D)$,
condition \eqref{ip4_mult} implies that
$B(\cdot, \cdot, Y)\in L^2(\Omega\times (a,b); \cL_2(U,H))$: hence, for every $X_a\in L^2(\Omega, \f_a, \Pi; H)$, 
thanks to Theorem \ref{th1} we know that there exist
\begin{gather}
  \label{X_ab}
  X_{a,b}\in L^2\left(\Omega; L^\infty(a,b; H)\right)\cap 
  L^p\left(\Omega\times(a,b); V\right)\,,\\
  \eta_{a,b}\in L^q\left(\Omega\times(a,b)\times D\right)^d\,, \qquad \xi_{a,b}\in L^1\left(\Omega\times(a,b)\times D\right)\,,
\end{gather}
such that $X_{a,b}$ is adapted with $\Pi$-almost surely weakly continuous paths in $H$ and the following relations hold:
\begin{gather}
    \label{var_form_ab}
    \begin{split}
    X_{a,b}(t)-\int_a^t{\Div\eta_{a,b}(s)\,ds}+\int_a^t{\xi_{a,b}(s)\,ds}=&X_a+\int_a^t{B(s, Y(s))\,dW_s} \quad\text{in } V_0^*\,,\\
    &\text{for every } t\in[a,b]\,, \quad\Pi\text{-almost surely}\,,\\
    \end{split}\\
    \eta_{a,b}\in\gamma(\nabla X_{a,b}) \quad\text{a.e.~in } \Omega\times(a,b)\times D\,,\\
    \xi_{a,b}\in\beta(X_{a,b}) \quad\text{a.e. in } \Omega\times(a,b)\times D\,,\\
    \label{j_l1_ab}
    j(X_{a,b})+j^*(\xi_{a,b})\in L^1\left(\Omega\times(a,b)\times D\right)\,,
\end{gather}
where $X_{a,b}$ is unique in the sense of Theorem \ref{th2}. Now, we need the following lemma.

\begin{lem}
  \label{lem_tau}
  For every $\tau>0$ and $n\in\En$ fixed, consider $X_{n\tau}\in L^2(\Omega, \f_{n\tau}, \Pi; H)$ and 
  $Y_1, Y_2\in L^2(\Omega\times (n\tau, (n+1)\tau)\times D)$ progressively measurable: then, if
  $(X_1, \eta_1, \xi_1)$ and $(X_2, \eta_2, \xi_2)$ are any respective solutions to \eqref{X_ab}--\eqref{j_l1_ab} with
  $a=n\tau$, $b=(n+1)\tau$ and same initial value $X_a=X_{n\tau}$, we have the following estimate:
  \beq
    \label{contr_est}
    \l|X_1-X_2\r|_{L^2(\Omega\times (n\tau, (n+1)\tau)\times D)}\leq
    \sqrt{\tau L_B}\l|Y_1-Y_2\r|_{L^2(\Omega\times (n\tau, (n+1)\tau)\times D)}\,.
  \eeq
\end{lem}
\begin{proof}
Taking the difference of equations \eqref{var_form_ab}
evaluated with $i=1,2$ and recalling the generalized It\^o formula \eqref{ito_for}, 
setting $X:=X_1-X_2$, $\eta:=\eta_1-\eta_2$ and $\xi:=\xi_1-\xi_2$, we easily get that for every $t\in[m\tau,(m+1)\tau]$
\[
  \begin{split}
  \frac{1}{2}\l|X(t)\r|^2_{L^2(\Omega\times D)}&+\int_0^t\int_{\Omega\times D}\eta(s)\cdot\nabla X(s)\,ds+
  \int_0^t\int_{\Omega\times D}{\xi(s)X(s)\,ds}\\
  &=\frac{1}{2}\l|B(Y_1)-B(Y_2)\r|^2_{L^2(\Omega\times(m\tau,(m+1)\tau); \cL_2(U,H))}\,.
  \end{split}
\]
Hence, using the Lipschitz continuity of $B$ and the monotonicity of $\gamma$ and $\beta$ we have
\[
  \frac{1}{2}\l|X_1-X_2\r|^2_{L^\infty(m\tau,(m+1)\tau; L^2(\Omega\times D))}
  \leq\frac{L_B}{2}\l|Y_1-Y_2\r|^2_{L^2(\Omega\times (m\tau,(m+1)\tau)\times D)}\,,
\]
from which \eqref{contr_est} follows.
\end{proof}

Now, let us build some solutions $X$, $\eta$ and $\xi$  in each sub-interval.
To this purpose, we choose $\tau>0$ such that the constant appearing in \eqref{contr_est} is less than 1, for example
\beq
  \label{tau}
  \tau:=\frac{1}{2L_B}\,.
\eeq

Firstly, we focus on $[0,\tau]$: taking into account the 
remarks that we have just made, it is well defined the function
\beq
  \label{phi_0}
  \Phi_0: L^2\left(\Omega\times(0,\tau)\times D\right)\rarr L^2\left(\Omega\times(0,\tau)\times D\right)\,,
  \quad \Phi_0(Y):=X\,,
\eeq
where $X$ is the unique solution to \eqref{X_ab}--\eqref{j_l1_ab} with the choices $a=0$ and $b=\tau$, with $X_0$ given by \eqref{ip1_mult}.
It is clear that
$X$ is a solution of problem \eqref{var_form_mult} in $[0,\tau]$ if and only if
it is a fixed point for $\Phi_0$. Thanks to the estimate \eqref{contr_est} and the choice \eqref{tau},
$\Phi_0$ is a contraction: hence, it has a fixed point
\[
  X^{(0)}\in L^2\left(\Omega;L^\infty(0,\tau; H)\right)\cap L^p\left(\Omega\times(0,\tau); V\right)\,,
\]
with $\Pi$-almost surely weakly continuous paths in $H$, which solves \eqref{var_form_mult} with certain 
\[
  \eta^{(0)}\in L^q\left(\Omega\times(0,\tau)\times D\right)^d\,, \qquad
  \xi^{(0)}\in L^1\left(\Omega\times(0,\tau)\times D\right)\,.
\]

Secondly, let us focus on $[\tau,2\tau]$, set $X_\tau:=X^{(0)}(\tau)$ (which is in $L^2(\Omega, \f_{\tau}, \Pi; H)$ since $X^{(0)}$ is adapted)
and define the function
\beq
  \label{phi_1}
  \Phi_1: L^2\left(\Omega\times(\tau,2\tau)\times D\right)\rarr L^2\left(\Omega\times(\tau,2\tau)\times D\right)\,,
  \quad \Phi_1(Y):=X\,,
\eeq
where $X$ is the solution to \eqref{X_ab}--\eqref{j_l1_ab} with the choices $a=\tau$ and $b=2\tau$.
As we have already done, thanks to the estimate \eqref{contr_est} and the choice \eqref{tau},
$\Phi_1$ is a contraction: hence, it has a fixed point
\[
  X^{(1)}\in L^2\left(\Omega;L^\infty(\tau, 2\tau; H)\right)\cap
     L^p\left(\Omega\times(\tau,2\tau); V\right)\,,
\]
with $\Pi$-almost surely weakly continuous paths in $H$,
which is a solution of \eqref{var_form_mult} with certain 
\[
  \eta^{(1)}\in L^q\left(\Omega\times(\tau,2\tau)\times D\right)^d\,, \qquad
  \xi^{(1)}\in L^1\left(\Omega\times(\tau,2\tau)\times D\right)\,.
\]

Suppose by induction that we have built
$(X^{(0)},\eta^{(0)}, \xi^{(0)}), \ldots, (X^{(m-1)}, \eta^{(m-1)}, \xi^{(m-1)})$
and let us show how to obtain $(X^{(m)},\eta^{(m)}, \xi^{(m)})$. We focus on the interval $[m\tau, (m+1)\tau]$,
set $X_{m\tau}:=X^{(m-1)}(m\tau)$ (which is in $L^2(\Omega, \f_{m\tau}, \Pi; H)$ since $X^{(m-1)}$ is adapted)
 and define the function
\beq
  \label{phi_m}
  \Phi_m: L^2\left(\Omega\times(m\tau, (m+1)\tau)\times D\right)\rarr L^2\left(\Omega\times(m\tau, (m+1)\tau)\times D\right)\,,
\eeq
which maps $Y$ into $X$, where $X$ is the solution to \eqref{X_ab}--\eqref{j_l1_ab} with the choices $a=m\tau$ and $b=(m+1)\tau$.
Now, $\Phi_m$ is a contraction thanks to \eqref{contr_est} and \eqref{tau}, so it has a fixed point
\[
  X^{(m)}\in L^2\left(\Omega;L^\infty(m\tau,(m+1)\tau; H)\right)
  \cap
 L^p\left(\Omega\times(m\tau,(m+1)\tau); V\right)\,:
\]
with $\Pi$-almost surely weakly continuous paths in $H$, which is a solution of \eqref{var_form_mult} with certain 
\[
  \eta^{(m)}\in L^q\left(\Omega\times(m\tau, (m+1)\tau)\times D\right)^d\,, \qquad
  \xi^{(m)}\in L^1\left(\Omega\times(m\tau,(m+1)\tau)\times D\right)\,.
\]

In this way, we can define the triplet $(X, \eta, \xi)$ by
setting $(X,\eta, \xi):=(X^{(m)},\eta^{(m)}, \xi^{(m)})$ in $\Omega\times[m\tau,(m+1)\tau)\times D$
for every $m\in\En$ until we reach $T$: bearing in mind how we have built
$(X^{(m)}, \eta^{(m)}, \xi^{(m)})$, it is clear that $X$, $\eta$ and $\xi$ are well-defined and satisfy conditions 
\eqref{sol1}--\eqref{sol3}, \eqref{incl_eta}--\eqref{j_integr} and \eqref{var_form_mult}.

Finally, if we do not assume \eqref{ip3}, it is clear that, using the same argument,
the respective solutions constructed in this way are well-defined 
and satisfy conditions \eqref{sol1_weak} and \eqref{var_form_mult_weak}
instead of \eqref{sol1} and \eqref{var_form_mult}, respectively.

\subsection{Continuous dependence on the initial datum}
We present here the proof of the proof of the continuous dependence results
contained in the last part of Theorem \ref{thm3}.
Here, we repeat exactly the same argument of Section \ref{cont_dep} with the choices
$B_1:=B(\cdot, X_1)$ and $B_2:=(\cdot, X_2)$.

If \eqref{ip3} is assumed, for any given $\tau>0$, the same computations on the interval $(0,\tau)$ get us to
\[ 
  \l|X_1-X_2\r|^2_{L^2(\Omega; L^\infty(0,\tau; H))} \lesssim
   \l|X_0^1-X_0^2\r|^2_{L^2(\Omega; H)} + \l|B(X_1)-B(X_2)\r|^2_{L^2(\Omega\times(0,\tau); \cL_2(U,H))}\,,
\]
so that using the Lipschitz continuity of $B$ we obtain
\[
  \l|X_1-X_2\r|^2_{L^2(\Omega; L^\infty(0,\tau; H))} \lesssim \l|X_0^1-X_0^2\r|^2_{L^2(\Omega; H)}
  + \tau\l|X_1-X_2\r|^2_{L^2(\Omega; L^\infty(0,\tau; H))}\,.
\]
Hence, choosing for example $\tau=\frac12$, we get the desired relation on the interval $[0,\tau]$.
The idea is clearly to iterate the procedure on the following intervals $[\tau, 2\tau]$, $[2\tau, 3\tau]$, $\ldots$
until we reach the final time $T$, so that \eqref{var_dep_mult} is proved. 
The important point that we have to check is that the choice of $\tau$ can be made
uniformly with respect to each sub-interval, but this is not difficult: as a matter of fact,
for any $n\geq1$, performing the same
computations on $[n\tau, (n+1)\tau]$ we obtain
\[
  \begin{split}
  \l|X_1-X_2\r|^2_{L^2(\Omega; L^\infty(n\tau,(n+1)\tau; H))} &\lesssim
  \l|X_1(n\tau)-X_2(n\tau)\r|^2_{L^2(\Omega; H)}\\
  &+ \tau\l|X_1-X_2\r|^2_{L^2(\Omega; L^\infty(n\tau,(n+1)\tau; H))}\,,
   \end{split}
\]
from which we deduce that the choice of $\tau$ is independent of $n$, and one can easily conclude by induction on $n$.
As we did in Section \ref{cont_dep}, if $X_0^1=X_0^1$, then by \eqref{var_dep_mult}
we have $X_1=X_2$, and hypothesis \eqref{ip3}
also ensures $\eta_1=\eta_2$ and $\xi_1=\xi_2$.

Secondly, if \eqref{ip3} is not assumed, proceeding as in the final part of Section \ref{cont_dep} we get
for every $t\in[0,T]$ that
\[
  \l|X_1(t)-X_2(t)\r|^2_{L^2(\Omega; H)}\lesssim
  \l|X_0^1-X_0^2\r|_{L^2(\Omega; H)}^2+
  \int_0^t{\l|(B(X_1)-B(X_2))(s)\r|_{L^2(\Omega; \cL_2(U,H))}^2\,ds}\,,
\]
from which \eqref{var_dep_mult_weak} follows using the Lipschitz continuity of $B$ and the
Gronwall lemma. Finally, if $X_0^1=X_0^1$, then  by \eqref{var_dep_mult_weak} $X_1=X_2$
and consequently $-\Div\eta_1+\xi_1=-\Div\eta_2+\xi_2$.


\appendix
\section{An integration-by-parts formula}
\label{A_app}

The aim of this Appendix is to give a complete proof of the generalized testing formula 
contained in equation \eqref{test_eps}:
throughout the section, we assume to work with the notations and
setting of Section \ref{first_lim}. Here, $\eps\in(0,1)$ and $\omega\in\Omega$ are fixed as usual.

The main point is that we cannot directly test equation \eqref{app_eps}
by $X^\eps-W^\eps_B$, as we did in Section \ref{first}, since the regularity of $X^\eps$ is not enough: more specifically,
$\partial_t(X^\eps-W^\eps_B)$ is only intended in $V_0^*$ and
we would need that $X^\eps-W^\eps_B$ takes values in $V_0$, but this is not the case. However, 
by condition \eqref{liminf_xi_eps} and the regularities of $X^\eps$, $W^\eps_B$ and $\eta^\eps$,
all the terms in \eqref{test_eps}
make sense: hence, the intuitive idea is that \eqref{test_eps} holds at least in a formal way. 
To give a rigorous proof of it, a natural way could be to try to pass to the limit as $\lambda\searrow0$ in \eqref{test_lam_eps}:
however, it is not necessarily true in our framework that equation \eqref{test_lam_eps} converges to
\eqref{test_eps} as $\lambda\searrow0$,
so this approach does not work. Hence, the idea is to
see \eqref{test_eps} as a limit problem as $\delta\searrow0$, for another 
parameter $\delta$, such that the approximations in $\delta$ have good smoothing properties
and behave better that the approximations in $\lambda$.
In this sense, a similar approach was presented in \cite{barbu-prato-rock} and \cite{mar-scar}, where 
the approximations were built
using suitable powers of the resolvent of the laplacian. However,
in our case we have to approximate also elements in $W^{-1, q}(D)$ (namely, $-\Div\eta^\eps$) and the resolvent of the laplacian does not work
since $-\Delta$ is not coercive on $V$: the idea is thus to 
identify another suitable space, in which \eqref{app_eps} can be intended,
and to define appropriate approximations on it.
To this purpose, we need some preparatory work.

First of all, note that the operator $-\Div:L^q(D)^d\rarr V^*$ is linear, continuous and
satisfies $\l|-\Div u\r|_{V^*}\leq\l|u\r|_{L^q(D)}$ for every $u\in L^q(D)^d$. Let us define the space
\[
  V^*_{div}:=\left\{-\Div u: u\in L^q(D)^d\right\}\subseteq V^*\,.
\]

Secondly, we introduce the space $V^*_{div}\oplus L^1(D)$ as 
the subspace of $V_0^*$ given by all the formal linear combinations
of elements in $V^*_{div}$ and $L^1(D)$.
With this notations, we can note that equation \eqref{app_eps} actually holds in $V^*_{div}\oplus L^1(D)$:
in other words, for every $t\in[0,T]$, we have 
\beq
  \label{app_eps'}
  \begin{split}
    \left(X^\eps-W^\eps_B\right)(t)+\int_0^t{\left(-\Div\eta^\eps(s)+\xi^\eps(s)\right)\,ds}=X_0 \quad
    \text{in } V^*_{div}\oplus L^1(D)\,.
  \end{split}
\eeq
Hence, the idea is that it is sufficient to identify a way to approximate only elements in $V^*_{div}\oplus L^1(D)$,
and not any element of $V_0^*$, which would be much more demanding.

To this end, for every $\delta\in(0,1)$, let $\rsz_\delta:=(I-\delta\Delta)^{-1}$
be the resolvent of the Laplace operator. It is well-known that for every $r\in[1,+\infty)$,
$\rsz_\delta:L^r(D)\rarr L^r(D)$ is a linear contraction converging to the identity as $\delta\searrow0$
in the strong operator topology
(the reader can refer to \cite{barbu_monot, ben-brez-cran, brez-str}).
In this setting, we define the operator $\mathbf{R}_\delta:L^r(D)^d\rarr L^r(D)^d$ extending $\rsz_\delta$ component-by-component:
consequently, we easily deduce that also $\mathbf{R}_\delta$ is a linear contraction on $L^r(D)^d$ converging to the identity
as $\delta\searrow0$. With this notations, we have the following result.
\begin{lem}
\label{lam_del_lem}
  For every $u\in L^q(D)^d$ such that $-\Div u\in L^1(D)$ (in the distributional sense), we have
  \[
    -\Div\mathbf{R}_\delta u=\rsz_\delta\left(-\Div u\right)\,.
  \]
  Moreover, for every $f\in H^1(D)$, we have
  \[
    \nabla \rsz_\delta f=\mathbf{R}_\delta\nabla f\,.
  \]
\end{lem}
\begin{proof}
Let us first assume that $u\in \left(C^\infty(D)\right)^d$:
then, using the definition of $\mathbf{R}_\delta$ and $\mathcal{R}_\delta$, integration by parts
and the fact that $\rsz_\delta$ commutes with $\Delta$,
for every $\varphi\in C^\infty_c(D)$ we have
\[
  \begin{split}
  \int_D{\left(-\Div u\right)\varphi}&=\int_D{u\cdot\nabla\varphi}=
  \sum_{i=1}^d\int_D{u_i\frac{\partial\varphi}{\partial x_i}}=
  \sum_{i=1}^d\int_D\left(\rsz_\delta u_i-\delta\Delta\rsz_\delta u_i\right)\frac{\partial\varphi}{\partial x_i}\\
  &=\int_D\mathbf{R}_\delta u\cdot\nabla\varphi+\delta\int_D\Delta(\Div\mathbf{R}_\delta u)\varphi
  =\int_D\left[-\Div\mathbf{R}_\delta u-\delta\Delta(-\Div\mathbf{R}_\delta u)\right]\varphi\,.
  \end{split}
\]
Hence, by definition of the resolvent, we deduce that $-\Div{\mathbf{R}_\delta u}=\rsz_\delta(-\Div u)$ 
for every $u\in \left(C^\infty(D)\right)^d$. At this point, if $u\in L^q(D)^d$ and $-\Div u\in L^1(D)$,
the first thesis follows by approximating $u$ with a sequence $\{u_n\}_{n\in\En}\subseteq \left(C^\infty(D)\right)^d$
such that $u_n\rarr u$ in $L^q(D)^d$ and $-\Div u_n\rarr -\Div u$ in $L^1(D)$. 
Finally, in a similar way, the second assertion is clearly true for every $f\in C^\infty(D)$: hence, 
given $f\in H^1(D)$, we can conclude by density approximating $f$ with a sequence $\{f_n\}_{n\in\En}\subseteq C^\infty(D)$.
\end{proof}

Now, for every $\delta\in(0,1)$, we introduce the operator
\[
  \Lambda_\delta^1:V^*_{div}\rarr V^*_{div}
\]
in the following way: for any given $f\in V^*_{div}$, with $f=-\Div u$ for a certain $u\in L^q(D)^d$, we set
$\Lambda_\delta^1f:=-\Div\mathbf{R}_\delta u$. Note that $\Lambda_\delta^1$ is well-defined:
indeed, if $f=-\Div u_1=-\Div u_2$, we have $-\Div(u_1-u_2)=0$ and by Lemma \ref{lam_del_lem}
we deduce that $0=\rsz_\delta(-\Div(u_1-u_2))=-\Div(\mathbf{R}_\delta(u_1-u_2))$,
so that $-\Div\mathbf{R}_\delta u_1=-\Div\mathbf{R}_\delta u_2$.
Secondly, we set
\[
  \Lambda_\delta^2:L^1(D)\rarr L^1(D)\,, \quad \Lambda^2_\delta:=\rsz_\delta\,.
\]
The first part of Lemma \ref{lam_del_lem} ensures that $\Lambda_\delta^1=\Lambda_\delta^2$ on 
the intersection $V^*_{div}\cap L^1(D)$:
hence, it is well-defined the operator
\beq
  \label{lam_del}
  \Lambda_\delta:=\Lambda_\delta^1\oplus\Lambda_\delta^2:V^*_{div}\oplus L^1(D)\rarr V^*_{div}\oplus L^1(D)
\eeq
such that
\beq
  \label{lam_del'}
  \Lambda_\delta(-\Div u)=-\Div\mathbf{R}_\delta u\,, 
  \quad \Lambda_\delta(f)=\rsz_\delta f \quad\text{for every } u\in L^q(D)^d\,,\; f\in L^1(D)\,,
\eeq
which is automatically linear.

We are now ready to build the approximations.
First of all, we choose $k\in\En$ as in \eqref{V0}, so that the $k$-th power $\rsz_\delta^k$
maps $H$ into $V_0\subseteq V\cap L^\infty(D)$.
At this point, 
we define
\beq
  \label{approx_del}
  X_\delta^\eps:=\rsz^k_\delta X^\eps\,, \quad W_\delta^\eps:=\rsz^k_\delta W^\eps_B\,, \quad
  \eta_\delta^\eps:=\mathbf{R}^k_\delta \eta^\eps\,, \quad \xi_\delta^\eps:= \rsz^k_\delta\xi^\eps\,, \quad
  X_{0}^\delta:=\rsz^k_\delta X_{0}\,:
\eeq
then, taking into account the properties of $\rsz_\delta$ and $\mathbf{R}_\delta$ and the second part of Lemma \ref{lam_del_lem},
thanks to conditions \eqref{eta_eps}--\eqref{xi_eps}, \eqref{X_eps_bis} and \eqref{cont_X_eps'}
we have as $\delta\searrow0$ that
\begin{gather}
  \label{del1}
  X_\delta^\eps(t)\rarr X^\eps(t) \quad\text{in } H \quad\text{for every } t\in[0,T]\,, 
  \quad X^\eps_\delta\rarr X^\eps \quad\text{in } L^p\left(0,T;V\right)\\
  \label{del2}
  W_\delta^\eps(t)\rarr W^\eps(t) \quad\text{in } H \quad\text{for every } t\in[0,T]\,,
  \quad W^\eps_\delta\rarr W^\eps_B \quad\text{in } L^p\left(0,T;V\right)\,,\\
  \label{del3}
  \eta_\delta^\eps\rarr \eta^\eps \quad\text{in } L^q\left(Q\right)^d\,, \qquad \xi_\delta^\eps\rarr\xi^\eps \quad\text{in } L^1(Q)\,,\\
  \label{del4}
  X^\delta_{0}\rarr X_{0} \quad\text{in } H\,.
\end{gather}
Now, applying the operator $\Lambda^k_\delta$
to equation \eqref{app_eps'}, we get for every $t\in[0,T]$ that
\beq
  \label{app_eps_del}
  \left(X_\delta^\eps-W_\delta^\eps\right)(t)-\int_0^t{\Div\eta^\eps_\delta(s)\,ds}+\int_0^t{\xi^\eps_\delta(s)\,ds}=
  X^\delta_{0}\,.
\eeq
With our choice of $k$, it now makes sense to test by $X_\delta^\eps-W_\delta^\eps$: it easily follows that
\beq
  \label{test_del}
  \frac{1}{2}\l|X^\eps_\delta(T)-W_\delta^\eps(T)\r|^2_H+\int_Q{\nabla\eta_\delta^\eps\cdot\nabla\left(X_\delta^\eps-W_\delta^\eps\right)}
  +\int_Q{\xi_\delta^\eps\left(X_\delta^\eps-W_\delta^\eps\right)}=\frac{1}{2}\l|X^{\delta}_{0}\r|^2_H\,,
\eeq
from which, taking into account \eqref{del1}--\eqref{del4}, we deduce that
\beq
  \label{lim_eps_del}
  \lim_{\delta\searrow0}\int_Q{\xi_\delta^\eps\left(X_\delta^\eps-W_\delta^\eps\right)}
  =\frac{1}{2}\l|X_{0}\r|^2_H-\frac{1}{2}\l|(X^\eps-W_B^\eps)(T)\r|_H^2-
  \int_Q{\nabla\eta^\eps\cdot\nabla\left(X^\eps-W_B^\eps\right)}\,.
\eeq
In order to evaluate the limit in the previous expression,
we take advantage of Vitali convergence theorem:
to this purpose,
thanks to \eqref{del1}--\eqref{del3}, we can assume with no restriction that 
$\xi^\eps_\delta\rarr\xi^\eps$ and $X_\delta^\eps-W_\delta^\eps\rarr X^\eps-W_B^\eps$ almost everywhere in $Q$.
Let us show that $\{\xi_\delta^\eps(X^\eps_\delta-W^\eps_\delta)\}_{\delta\in(0,1)}$ is uniformly integrable in $Q$:
by conditions \eqref{convex3}--\eqref{j_even} and thanks to the generalized Jensen inequality
for the positive operator $R_\delta$ (see \cite{jen1, jen2} for references), we have
\[
  \begin{split}
  \pm\xi^\eps_\delta(X^\eps_\delta-W^\eps_\delta)&\leq
  j\left(\pm(X^\eps_\delta-W^\eps_\delta)\right)+j^*(\xi^\eps_\delta)=
  j\left(X^\eps_\delta-W^\eps_\delta\right)+j^*(\xi^\eps_\delta)\\
  &\leq \rsz^k_\delta\left[j\left(X^\eps-W^\eps_B\right)+j^*(\xi^\eps)\right] \quad\text{a.e.~in } Q\,.
  \end{split}
\]
Now, since $j(X^\eps-W^\eps_B), j^*(\xi^\eps)\in L^1(Q)$ thanks to \eqref{j_l1} and \eqref{stoc_reg'},
the right hand side of the previous expression
converges in $L^1(Q)$ and consequently it is uniformly integrable in $Q$: we deduce that also 
$\{\xi^\eps_\delta(X^\eps_\delta-W^\eps_\delta)\}_{\delta\in(0,1)}$ is uniformly integrable in $Q$.
Hence, by Vitali convergence theorem, we infer that
\[
  \xi^\eps_\delta\left(X^\eps_\delta-W^\eps_\delta\right)\rarr \xi^\eps\left(X^\eps-W^\eps_B\right) \quad\text{in } L^1(Q) \quad\text{as }
  \delta\searrow0\,,
\]
so that passing to the limit in \eqref{lim_eps_del} we recover exactly \eqref{test_eps}.

\section{The generalized It\^o formula}
\label{B_app}
In this appendix, we prove a generalized It\^o formula, which is widely used in Sections~\ref{second_lim} and \ref{cont_dep}:
we collect the general result in the following proposition. Throughout the section, we assume the general setting 
\eqref{spaces}--\eqref{rel_V0}.

\begin{prop}
  \label{prop_ito}
  Assume the following conditions:
  \begin{gather}
    \label{ito1}
    Y_0\in L^2\left(\Omega, \f_0, \Pi; H\right)\,,\\
    \label{ito2}
    T\in L^2\left(\Omega\times(0,T); \cL_2(U,H)\right) \quad\text{progressively measurable}\,,\\
    \label{ito3}
    Y\in L^2\left(\Omega; L^\infty(0,T; H)\right)\cap L^p\left(\Omega\times(0,T); V\right)\,,\qquad
    Y\in C^0_w\left([0,T]; H\right) \quad\Pi\text{-a.s.}\,,\\
    \label{ito4}
    f\in L^q\left(\Omega\times(0,T)\times D\right)^d\,, \qquad 
    g\in L^1\left(\Omega\times(0,T)\times D\right)\,,\\
    \label{ito5}
    \exists\;\alpha>0\,:\quad j(\alpha Y)+j^*(\alpha g)\in L^1\left(\Omega\times (0,T)\times D\right)\,,\\
    \label{ito6}
    Y(t)-\int_0^t{\Div f(s)\,ds}+\int_0^t{g(s)\,ds}=Y_0+\int_0^t{T(s)\,dW_s} \quad\text{in } V_0^*
  \end{gather}
  for every $t\in[0,T]$, $\Pi$-almost surely.
  Then, the following It\^o formula holds
  \beq
    \label{ito_for}
    \begin{split}
    \frac{1}{2}\l|Y(t)\r|_H^2&+\int_0^t\int_{D}{f(s)\cdot\nabla Y(s)\,ds}+
    \int_0^t\int_{D}{g(s)Y(s)\,ds}\\
    &=\frac{1}{2}\l|Y_0\r|_H^2+\frac{1}{2}\int_0^t{\l|T(s)\r|_{\cL_2(U,H)}^2\,ds}+\int_0^t\left(Y(s), T(s)\,dW_s\right)
    \end{split}
  \eeq
  for every $t\in[0,T]$, $\Pi$-almost surely.
  Furthermore, if hypothesis \eqref{ito3} is replaced by the weaker condition
  \beq
  \label{ito3bis}
  Y\in L^\infty\left(0,T; L^2(\Omega; H)\right)\cap L^p\left(\Omega\times(0,T); V\right)\cap
  C^0_w\left([0,T]; L^2(\Omega; H)\right)\,,
  \eeq
  then instead of \eqref{ito_for} we have the following for every $t\in[0,T]$:
  \beq
    \label{ito_for'}
    \begin{split}
    \frac{1}{2}\l|Y(t)\r|_{L^2(\Omega; H)}^2&+\int_0^t\int_{\Omega\times D}{f(s)\cdot\nabla Y(s)\,ds}+
    \int_0^t\int_{\Omega\times D}{g(s)Y(s)\,ds}\\
    &=\frac{1}{2}\l|Y_0\r|_{L^2(\Omega; H)}^2+
    \frac{1}{2}\int_0^t{\l|T(s)\r|_{L^2(\Omega; \cL_2(U,H))}^2}\,ds\,.
    \end{split}
  \eeq
\end{prop}
\begin{proof}
  We proceed exactly in the same way as in Appendix \ref{A_app}. If 
  $k$ is given by \eqref{V0} and for every $\delta\in(0,1)$, $\rsz_\delta$ and $\mathbf{R}_\delta$ are as in Appendix \ref{A_app},
  we define
  \[
  Y_\delta:=\rsz^k_\delta Y\,, \quad T_\delta:=\rsz^k_\delta T\,, \quad
  f_\delta:=\mathbf{R}^k_\delta f\,, \quad g_\delta:= \rsz^k_\delta g\,, \quad
  Y_{0}^\delta:=\rsz^k_\delta Y_{0}\,:
  \]
  hence, thanks to \eqref{ito1}--\eqref{ito4} and Lemma \ref{lam_del_lem} we have as $\delta\searrow0$
  \begin{gather}
  \label{del1'}
  Y_\delta(t)\rarr Y(t) \quad\text{in } H \quad\text{for every } t\in[0,T]\,, \quad\Pi\text{-almost surely}\,,\\
  \label{del1'bis}
  Y_\delta\rarr Y \quad\text{in } L^p\left(\Omega\times (0,T);V\right)\\
  \label{del2'}
  T_\delta \rarr T \quad\text{in } L^2\left(\Omega\times(0,T); \cL_2(U,H)\right)\,,\\
  \label{del3'}
  f_\delta \rarr f \quad\text{in } L^q\left(\Omega\times Q\right)^d\,, \qquad g_\delta\rarr g \quad\text{in } L^1(\Omega\times Q)\,,\\
  \label{del4'}
  Y^\delta_{0}\rarr Y_{0} \quad\text{in } L^2(\Omega; H)\,.
  \end{gather}
  Consequently, if we apply the operator $\Lambda_\delta^k$ to \eqref{ito6}, taking definition \eqref{lam_del}--\eqref{lam_del'}
  into account, we have $\Pi$-almost surely that
  \[
    Y_\delta(t)-\int_0^t{\Div f_\delta(s)\,ds}+\int_0^t{g_\delta(s)\,ds}=Y^\delta_0+\int_0^t{T_\delta(s)\,dW_s} \quad\text{in } H\,,
    \quad\text{for every } t\in[0,T]\,.
  \]
  Now, with our choice of $k$, we can apply the classical It\^o formula (see \cite[Thm.~4.2.5]{prevot-rock} for example) to recover
  that $\Pi$-almost surely, for every $t\in[0,T]$,
  \beq
  \label{ito_approx}
  \begin{split}
    \frac{1}{2}\l|Y_\delta(t)\r|_H^2&+\int_0^t\int_D{f_\delta(s)\cdot\nabla Y_\delta(s)\,ds}+\int_0^t\int_D{g_\delta(s)Y_\delta(s)\,ds}\\
    &=\frac{1}{2}\l|Y^\delta_0\r|_H^2+\frac{1}{2}\int_0^t{\l|T_\delta(s)\r|_{L_2(U,H)}^2\,ds}+\int_0^t{\left(Y_\delta(s), T_\delta(s)\,dW_s\right)}\,.
    \end{split}
  \eeq
  Now, let us focus on the stochastic integral: we have
  \[
  \begin{split}
  &\int_0^t\left(Y_\delta(s), T_\delta(s)\,dW_s\right)-\int_0^t\left(Y(s), T(s)\,dW_s\right)\\
  &\qquad\qquad=\int_0^t\left(Y_\delta(s), (T_\delta-T)(s)\,dW_s\right)+\int_0^t\left((Y_\delta-Y)(s), T(s)\,dW_s\right)\,,
  \end{split}
  \]
  where thanks to the Davis inequality and \eqref{del1'bis}--\eqref{del2'} we have
  \[
  \begin{split}
  \exval\sup_{t\in[0,T]}\left|\int_0^t\left(Y_\delta(s), (T_\delta-T)(s)\,dW_s\right)\right|&\lesssim
  \exval\left[\left(\int_0^T{\l|Y_\delta(s)\r|^2_H\l|(T_\delta-T)(s)\r|^2_{\cL_2(U,H)}\,ds}\right)^{1/2}\right]\\
  &\lesssim\l|T_\delta-T\r|_{L^2(\Omega\times(0,T); \cL_2(U,H))}\rarr0
  \end{split}
  \]
  and, by the dominated convergence theorem, also
  \[
  \begin{split}
  \exval\sup_{t\in[0,T]}\left|\int_0^t\left((Y_\delta-Y)(s), T_s\,dW_s\right)\right|^2&\lesssim
  \exval\left[\left(\int_0^T{\l|(Y_\delta-Y)(s)\r|^2_H\l|T_s\r|^2_{\cL_2(U,H)}\,ds}\right)^{1/2}\right]\rarr0\,.
  \end{split}
  \]
  Hence, we deduce that $\int_0^\cdot(Y_\delta(s), T_\delta(s)\,dW_s)\rarr\int_0^\cdot(Y(s), T(s)\,dW_s)$ in
  $L^2(\Omega; L^\infty(0,T))$, so that consequently (at least for a subsequence)
  \[
    \int_0^t\left(Y_\delta(s), T_\delta(s)\,dW_s\right)\rarr\int_0^t\left(Y(s), T(s)\,dW_s\right) \quad\text{for every } t\in[0,T]\,, 
    \quad\Pi\text{-almost surely}\,.
  \]
  Hence, letting $\delta\searrow0$ and taking into account \eqref{del1'}--\eqref{del4'}, $\Pi$-almost surely we have
  \beq
  \label{lim_del}
  \begin{split}
  \lim_{\delta\searrow0}\int_{(0,t)\times D}{g_\delta Y_\delta}&=\frac{1}{2}\l|Y_0\r|^2_H
  +\frac{1}{2}\int_0^{t}{\l|T(s)\r|_{\cL_2(U,H)}^2\,ds}+\int_0^t\left(Y(s), T(s)\,dW_s\right)\\
  &-\frac{1}{2}\l|Y(t)\r|^2_{H}-\int_{(0,t)\times D}{f_\delta\cdot\nabla Y}
  \quad\text{for every } t\in[0,T]\,:
  \end{split}
  \eeq
  we evaluate the limit on the left hand side using Vitali's convergence theorem.
  To this purpose, by \eqref{del1'} and \eqref{del3'} we can assume with no restriction that $Y_\delta\rarr Y$ and $g_\delta\rarr g$
  almost everywhere in $\Omega\times(0,t)\times D$; 
  moreover, thanks to conditions \eqref{convex3}--\eqref{j_even} and the generalized Jensen inequality
  for positive operators (see \cite{jen1, jen2}), we have
  \[
  \begin{split}
  \pm \alpha^2g_\delta Y_\delta&\leq
  j\left(\pm\alpha Y_\delta\right)+j^*\left(\alpha g_\delta\right)=
  j\left(\alpha Y_\delta\right)+j^*\left(\alpha g_\delta\right)
  \leq \rsz^k_\delta\left[j(\alpha Y)+j^*(\alpha g)\right]\,.
  \end{split}
  \]
  Thanks to \eqref{ito5} and the properties of $\rsz_\delta$, the term on the right hand side converges in $L^1(\Omega\times(0,t)\times D)$,
  hence it is uniformly integrable: consequently, we deduce that also
  $\{g_\delta Y_\delta\}_{\delta\in(0,1)}$ is
  uniformly integrable, and Vitali's convergence theorem implies that
  \[
  g_\delta Y_\delta\rarr gY \quad\text{in } L^1\left(\Omega\times(0,t)\times D\right)\,,   \quad\text{as } \delta\searrow0\,,
  \]
  so that passing to the limit in \eqref{lim_del} we obtain \eqref{ito_for}.\\
  To show \eqref{ito_for'}, we proceed in a very similar way: note that since \eqref{ito3}
  is replaced by \eqref{ito3bis}, then instead of \eqref{del1'} we have
  \[Y_\delta(t)\rarr Y(t) \quad\text{in } L^2(\Omega; H)\,, \quad\text{for every } t\in[0,T]\,.\]
  Once we have obtained \eqref{ito_approx} as before, we observe that
  the stochastic integral in \eqref{ito_approx} is a local martigale, so that
  there exists a sequence of increasing stopping times $\{\tau_n\}_{n\in\En}$ such that $\tau_n\nearrow\infty$
  and the corresponding stopped processes are martingales: hence, stopping \eqref{ito_approx} at time $\tau_n$, taking 
  expectations and then letting $n\rarr\infty$, thanks to dominated convergence theorem we directly obtain for every $t\in[0,T]$
  \[
  \begin{split}
    \frac{1}{2}\l|Y_\delta(t)\r|_{L^2(\Omega; H)}^2&+\int_0^t\int_{\Omega\times D}{f_\delta(s)\cdot\nabla Y_\delta(s)\,ds}+
    \int_0^t\int_{\Omega\times D}{g_\delta(s)Y_\delta(s)\,ds}\\
    &=\frac{1}{2}\l|Y_0^\delta\r|_{L^2(\Omega; H)}^2+
    \frac{1}{2}\int_0^t{\l|T_\delta(s)\r|_{L^2(\Omega; \cL_2(U,H))}^2}\,ds\,.
    \end{split}
  \]
  At this point, \eqref{ito_for'} follows as before letting $\delta\searrow0$ in the previous equation.
\end{proof}


\def\cprime{$'$}

\end{document}